\documentclass[10pt,twoside, a4paper, english, reqno]{amsart}
\usepackage[dvips]{epsfig}
\usepackage{amscd}
\usepackage{amssymb}
\usepackage{amsthm}
\usepackage{amsmath}
\usepackage{latexsym}

\usepackage{upref}
\usepackage{hyperref}
\usepackage{color}
\usepackage[usenames,dvipsnames]{xcolor}

\theoremstyle{plain}
\newtheorem{thm}{Theorem}[section]
\theoremstyle{plain}
\newtheorem{lem}[thm]{Lemma}
\newtheorem{prop}[thm]{Proposition}

\theoremstyle{definition}
\newtheorem{defi}{Definition}[section]
\newtheorem{rem}{Remark}[section]

\newenvironment{Assumptions}
{
\setcounter{enumi}{0}

\begin{enumerate}}
{\end{enumerate} }

\newcommand{\R}{\ensuremath{\mathbb{R}}}

\newcommand{\goto}{\ensuremath{\rightarrow}}
\newcommand{\grad}{\ensuremath{\nabla}}
\newcommand{\eps}{\ensuremath{\varepsilon}}

\numberwithin{equation}{section} \allowdisplaybreaks

\title[Stochastic optimal control and existence of invariant measure]
{Nonlinear SPDE driven by L\'{e}vy noise: Well-posedness, optimal control and invariant measure}

\date{}

\subjclass[2000]{45K05, 46S50, 49L20, 49L25, 91A23, 93E20}

\keywords{ Nonlinear stochastic PDE, Strong and martingale solution, Jakubowski’s version of Skorokhod theorem, Weak optimal control, Invariant measure.}

\author[ Kavin]{Kavin R}
\address[Kavin R] {\newline 
Department of Mathematics,
Indian Institute of Technology Delhi,
Hauz Khas, New Delhi, 110016, India.}
\email[] {maz198757@iitd.ac.in}

\author[Ananta K. Majee]{Ananta K. Majee}
\address[Ananta K. Majee]{\newline
Department of Mathematics, Indian Institute of Technology Delhi,
Hauz Khas, New Delhi-110016, India.}
\email[]{majee@maths.iitd.ac.in}

\thanks{}

\begin{document}
\begin{abstract}
In this article, we study a nonlinear stochastic control problem perturbed by multiplicative L\'{e}vy noise, where the nonlinear operator~(in divergence form) satisfies $p$-type growth with coercivity assumptions. By using Aldous tightness criteria and Jakubowski’s version of the Skorokhod theorem on non-metric spaces along with standard $L^1$-method, 
we establish existence of path-wise unique strong solution. Formulating the associated control problem, and using variational approach together with the convexity property of cost functional (in control variable), we establish existence of a weak optimal solution of the underlying problem. We use the technique of Maslowski and Seidler to prove existence of an invariant measure for uncontrolled SPDE driven with multiplicative L\'{e}vy noise. 
\end{abstract}

\maketitle

\section{Introduction}
Let $D\subset \R^d$ be a bounded domain with Lipschitz boundary $\partial D$,  $W(t)$ be a cylindrical Wiener process in a separable Hilbert space $L^2$, $N({\rm d}z,{\rm d}t)$ be a time-homogeneous Poisson random measure independent of $W$ \cite[pp. 631]{peszat,erika2009} on $({\tt E},\mathcal{E})$ with intensity measure $m({\rm d}z)$, defined on the filtered probability space
$(\Omega,  \mathcal{F},\mathbb{P}, \{ \mathcal{F}_t\}_{t \geq 0})$ satisfying the usual hypothesis,  where $({\tt E},\mathcal{E},m)$ is a $\sigma$-finite measure space. We are interested in the theory of existence of a weak optimal solution and invariant measure for the nonlinear stochastic control problem perturbed by L\'{e}vy noise:
 \begin{equation}\label{eq:nonlinear}
 \begin{aligned}
  {\rm d} u -\mbox{div}_x \big( {\tt A}(x, u, \grad u) + \vec{F}(u)\big)\,{\rm d}t &=  U\,{\rm d}t + {\tt h}(u)\,dW(t) + \int_{\tt E}\eta(x,u;z)\widetilde{N}({\rm d}z,{\rm d}t) \quad \text{in}\,\,\Omega \times D_T, \\
  u&=0 \quad \text{on}\,\, \Omega \times \partial D_T\,, \\
  u(0,\cdot)&=u_0(\cdot)\quad \text{in}\,\,\Omega \times  D\,,
  \end{aligned}
 \end{equation}
 where  $T>0$, $D_T=(0,T)\times D$ and $\partial D_T=(0,T)\times \partial D$. In \eqref{eq:nonlinear},  ${\tt h}: L^2 \mapsto \mathcal{L}_2(L^2)$ and $\eta:D\times \R \times {\tt E} \mapsto \R$
 are given noise coefficients and $$ \widetilde{N}({\rm d}z,{\rm d}t)= N({\rm d}z,{\rm d}t)-\, m({\rm d}z)\,{\rm d}t,$$ 
 the  time homogeneous compensated Poisson random measure.  Moreover, the nonlinear operator ${\tt A}: D\times \R \times \R^d \goto \R^d$ is a Carath\'{e}odory function satisfying coercivity and appropriate growth conditions ~(see Section \ref{sec:technical-framework} for the complete list of assumptions). Furthermore, $\vec{F}:\R \goto \R^d$ is a given Lipschitz continuous function, and $U$, serves as a control, is a given $L^2$-valued  $\{\mathcal{F}_t\}$-predictable process on the given stochastic basis. Furthermore, $W(t)$ can be expressed as
             \begin{equation*}
                 W(t) := \sum_{n=1}^{\infty} e_n \beta_n(t),
             \end{equation*}
             for some orthonormal basis $(e_n)_{n \in \mathbb{N}}$ of $L^2$ and a sequence of independent Brownian motions $(\beta_{n})_{n \in \mathbb{N}}$ on $\R$. Then $W$ can be interpreted as a $\mathcal{U}$-valued $Q$-Wiener process with covariance matrix $ {\rm diag} (\frac{1}{n^2})$, where $\mathcal{U}$ is the completion of $L^2$ with the scalar product~(see \cite[Section $4.1~\&~ 4.2$]{claud})
             \begin{equation*}
                (u,v)_\mathcal{U} := \sum_{n=0}^{\infty} \frac{(v,e_n)_2 (u,e_n)_2}{n^2}, \quad u , v \in L^2. 
             \end{equation*}   
The equation \eqref{eq:nonlinear} could be viewed as a stochastic perturbation of an evolutionary $p$-type growth with nonlinear sources.  This type of equations arise in many different fields e.g., in Mechanics, Physics, and Biology \cite{Dibenedetto1993, Wu2001}. Due to wide applications in physical contexts and also for more technical novelties, the study of well-posedness results for \eqref{eq:nonlinear} is more subtle. Because of nonlinear drift and diffusion terms in \eqref{eq:nonlinear}, it may not be possible to define a semi-group solution. Moreover, one cannot use the results of monotone or locally monotone SPDEs, see e.g.,  \cite{Liu2015,Pardoux1975} due to the presence of nonlinear perturbation of the type $-{\rm div} \vec{F}(u)$ for Lipschitz continuous $\vec{F}$.

In a recent articles \cite{Majee2020,Vallet2019}, the authors have considered stochastic evolutionary $p$-Laplace equation~($p>2$) perturbation by nonlinear sources. In \cite{Majee2020}, the author used pseudo monotonicity methods, Aldous tightness criterion and the Jakubowski-Skorokhod theorem \cite{Jakubowski1998} on a non-metric space to prove existence of a martingale solution. Moreover, by using variational approach, the author has shown the existence of a weak optimal solution to corresponding initial value control problem. In  \cite{Vallet2019}, existence of a martingale solution was shown by using  pseudo monotonicity methods along with the theorems of Skorokhod and Prokhorov. Moreover, the authors established well-posedness of strong solution   via path-wise uniqueness and Gy\"{o}ngy-Krylov characterization \cite{Krylov1996} of convergence in probability. The well-posedness results of the stochastic doubly nonlinear $p$-Laplace type equation~$(p>2)$ has been established in \cite{Wittbold2019, Majee2022}. Very recently, the authors in \cite{Vallet2021} have shown well-posedness of strong solution of \eqref{eq:nonlinear} in case of additive Brownian noise, where $p$- growth of the operator ${\tt A}$ satisfies the condition $\frac{2d}{d+1}<p< \infty$. By constructing an approximate solutions (via semi-implicit time discretization), 
and using the theorems of Skorokhod and Prokhorov,  first they have shown existence of a martingale solution and then using an argument of path-wise uniqueness and Gy\"{o}ngy-Krylov characterization the authors have established existence of a unique strong solution. 
\subsection{Aim and outline of the paper}
In this article, our main aim is to establish ${\rm i)}$ well-posedness of strong solution for \eqref{eq:nonlinear}, ${\rm ii)}$ existence of a weak optimal solution and ${\rm iii)}$ existence of an invariant measure in case of uncontrolled SPDE. We first construct an approximate solutions  $\tilde{u}_{\kappa}:=\big\{ \tilde{u}_{\kappa}(t); t\in [0,T]\}$~(cf.~\eqref{eq:defi-piecewise-affine-function}), and then derive its necessary {\it a-priori} estimates which is used to show  tightness of $\mathcal{L}(\tilde{u}_{\kappa})$, in some appropriate 
space via the Aldous condition~(see Definition \ref{defi:aldous-condition}).
Jakubowski's version of the Skorokhod theorem on a non-metric space is used to get
a new probability space $(\bar{\Omega}, \bar{\mathcal{F}}, \bar{\mathbb{P}})$ and sequence of processes $v_\kappa$ (cf. \eqref{defi:discrete-variables-new-prob-space}) and $v_*$ with $v_\kappa \goto v_* $ a.s. A Minty type monotonicity argument is used to identify the weak limit of the sequence $\big({\tt A}(x, v_\kappa, \grad v_\kappa))$~(cf.~{\bf step} ${\rm (ii)}$ of Subsection \ref{subsec:existence}) ---which then yields the existence of
a weak solution $\bar{\pi}:=\big( \bar{\Omega}, \bar{\mathcal{F}},  \bar{\mathbb{P}}, \bar{\mathbb{F}}, W_*, N_*, v_*, U_*\big)$~(cf.~{\bf step} ${\rm (iii)}$ of Subsection \ref{subsec:existence})
of \eqref{eq:nonlinear}. We use standard $L^1$-contraction principle to establish path-wise uniqueness of weak solutions. As a consequence, we are able to show existence of a unique strong solution. 

 We formulate the stochastic optimal control problem as follows: given a deterministic target profile ${\tt u}_{\rm det}$, the terminal payoff $
 {\tt \Psi}$, find a weak admissible solution tuple $\pi^*:=\big( \Omega^*, \mathcal{F}^*, \mathbb{P}^*,\{ \mathcal{F}_t^*\}, W^*,  N^*, u^*, U^*\big)$ such that it minimizes the cost functional
\begin{equation} \label{eq:control-problem}
 \begin{aligned}
  \mathcal{\pmb J}(\pi)= \mathbb{E}\Big[ \int_0^T\big( \|u(t)-{\tt u}_{\rm det}(t)\|_{L^2}^2 + \|U(t)\|_{L^2}^2\big)\,{\rm d}t + {\tt \Psi}(u(T))\Big] \\
  \text{with}~~\pi=\big( \Omega, \mathcal{F},\mathbb{P},\{ \mathcal{F}_t\}, W, N, u, U\big)~~\text{subject to}~~\eqref{eq:nonlinear}\,\,.
 \end{aligned}
\end{equation}
We use a standard variational method~(cf.~ \cite{Serrano,TMAV-2019,Majee2020}) to prove existence of weak optimal solution of \eqref{eq:control-problem}. 
We first consider a minimizing weak admissible solutions 
$\pi_n=\big( \Omega_n, \mathcal{F}_n,\mathbb{P}_n,\{ \mathcal{F}_t^n\}, W_n, N_n, u_n, U_n\big)$~(which exists thanks to Theorem \ref{thm:existence-weak}), and then follow the similar arguments as invoked in existence result together with the convexity property of the cost functional $\mathcal{\pmb J}$( with
respect to the control variable) to show existence of an optimal solution in the sense of Definition \ref{defi:optimal-control}.

 Finally, we devoted to establish the existence of invariant measure for the strong solution of \eqref{eq:nonlinear-no-control} by invoking the technique of Maslowski and Seidler \cite{Seidler1999}. We prove two conditions namely 
${\rm a)}$ boundedness in probability and ${\rm b)}$ sequentially weakly Feller property of the Markov semi-group $\{P_t\}_{t\ge 0}$~(cf.~\eqref{defi:semi-group}) with the help of  continuous dependency of martingale solution of \eqref{eq:nonlinear-no-control} on initial data (cf.~ Remark \ref{rem:cont-dependence-initial-cond}) and uniqueness in law. 
\vspace{.1cm}

The remaining part of the paper is organized as follows. We state the technical assumptions, define the notion of solutions
for the problems \eqref{eq:nonlinear} and \eqref{eq:control-problem}, and  state the main results of the paper in Section \ref{sec:technical-framework}. Section \ref{sec:existence-weak-solu} is devoted to establish the existence of a unique strong solution of \eqref{eq:nonlinear}.  The existence of a weak optimal solution of \eqref{eq:control-problem} is shown in Section \ref{thm:existence-optimal-control}. In the final section \ref{Existence of Invariant measure}, we prove the existence of an invariant measure.

 \section{Technical framework and statement of the main results}\label{sec:technical-framework}
In every part of this paper, we use the letters $C,\,K$ etc to denote various generic constants. Here and in the sequel, we denote by $(L^q, \|\cdot\|_{L^q}), q\in \mathbb{N}$ the standard spaces
 of $q^{th}$ order integrable functions on $D$. By $(W_0^{1,p}, \|\cdot\|_{W_0^{1,p}})$, we denote the standard Sobolev spaces on $D$ for $p\in \mathbb{N}$. The dual space of $W_0^{1,p}$ is denoted by $W^{-1,p^\prime}$ with duality pairing $\big\langle \cdot,\cdot\big\rangle$, where $p^\prime$ is the convex conjugate of $p$. By $\mathcal{L}_2(L^2)$, we denote the space of Hilbert-Schmidt operators from $L^2$ to $L^2$ on $D$.  We denote by $(\mathbb{X},w)$ the topological space $\mathbb{X}$ equipped with the weak topology.
 \vspace{.1cm}
 \subsection{Technical assumptions:} Throughout this article, we assume the following assumptions on the data.
 \begin{Assumptions}
\item \label{A1} $u_0 \in L^2$, and $U$ is a $L^2$-valued predictable process such that
$U\in L^2(\Omega; L^2(0,T;L^2))$. 
\item \label{A2} Let $p>2$ and ${\tt A}: D \times \R \times \R^d \goto \R^d$ be a Carath\'{e}odory function which satisfies the following conditions:
\begin{itemize}
\item [(i)] \underline{Monotonicity of ${\tt A}$:} for a.e. $x\in D$ and $\forall~\zeta_1, \zeta_2 \in \R^d, \theta \in \R$,
\begin{align*}
\big({\tt A}(x,\theta, \zeta_1)- {\tt A}(x,\theta, \zeta_2) \big)\cdot (\zeta_1- \zeta_2) \ge 0\,.
\end{align*}
\item [(ii)] \underline{ Coercivity of ${\tt A}$:} there exist a constant $C_1>0$ and a function $K_1 \in L^1$ such that 
\begin{align*}
 {\tt A} (x,\lambda, \zeta) \cdot \zeta \ge C_1 |\zeta|^p - K_1(x) \quad \text{for a.e. $x\in D$ and for all $\zeta \in \R^d, \lambda \in \R$}.
\end{align*}
\item [(iii)] \underline{ Boundedness of ${\tt A}$:} there exist constants $C_2, C_3 \ge 0$ and $0\le K_2 \ni L^{p^\prime}$ such that 
\begin{align*}
|{\tt A}(x, \lambda, \xi)| \le C_2 |\xi|^{p-1}  + C_3 |\lambda|^{p-1}+ K_2(x) \quad \text{ for a.e. $x\in D$ and  $ \forall~\xi \in \R^d, \lambda \in \R$}.
\end{align*}
\item [(iv)]\underline{ Lipschitzness of ${\tt A}$:} ${\tt A}$ is Lipschitz continuous in the second argument in the sense that there exist a constant $C_4 \ge 0$ and  $0\le K_3 \ni L^{p^\prime}$ such that for a.e. $x\in D$ and  $ \forall~\xi \in \R^d, \lambda_1, \lambda_2 \in \R$,
\begin{align*}
\big| {\tt A}(x,\lambda_1, \xi)- {\tt A}(x,\lambda_2, \xi)\big| \le \big( C_4 |\xi|^{p-1} + K_3(x)\big) |\lambda_1-\lambda_2|\,.
\end{align*}
\end{itemize}
 \item \label{A3} $\vec{F}:\R \mapsto \R^d$ is Lipschitz continuous with $\vec{F}(0)=\vec{0}$.
\item \label{A4} For each $u \in L^2$, the map ${\tt h}(u) : L^2 \goto L^2$  is defined by ${\tt h}(u) e_n =\{x\mapsto {\tt h}_n(u(x))\}$, where $\{e_n\}_n$ is an orthonormal basis of $L^2$ and each ${\tt h}_n(\cdot)$ is a regular function on $\R$. In particular, we assume that ${\tt h}_n \in C(\R)$ satisfies the following bounds: for $\xi, \zeta \in \R$
\begin{align*}
    \sum_{n \ge 1} | {\tt h}_n (\xi)|^2 \le  C_5(1+|\xi|^2), \quad 
    \sum_{n \ge 1} |{\tt h}_n (\xi) - {\tt h}_n(\zeta)|^2 \le L_\sigma |\xi-\zeta|^2
\end{align*}
for some positive constants $L_\sigma$ and $C_5$. Define $ \displaystyle {\tt h}(u):= \sum_{n \ge 1}^{} {\tt h}_n (u(\cdot)) e_n$. 
 An immediate consequence of the above assumption yields 
\begin{align*}
 \|{\tt h}(v) \|_{\mathcal{L}_2(L^2)}^{2} \le C_5 (|D| + \|v\|_{L^2}^{2})\quad \forall\, v \in L^2\,. 
\end{align*}
\item \label{A5} ${\tt E}= {\tt O} \times \R^{*}$ for some ${\tt O}$, a subset of the Euclidean space. The Borel measure $m$ on ${\tt E}$ has the form $\lambda \times \mu$, where $\lambda$ is a Radon measure on ${\tt O}$ and $\mu$ is so-called one dimensional L\'{e}vy measure on $\R^{*}$.
\item \label{A6}  There exist $\gamma(z) \in L^2({\tt E},m)$ with $0 \le \gamma(z) \le 1$, 
constants $\lambda^*\in [0,1), L_\eta >0$, and a non-negative function $g\in L^\infty$ such that
 for all $z \in {\tt E},~\zeta, \xi \in \R$ and $x,y \in D$,
 \begin{align*}
 |\eta(x,\zeta;z)| \le g(x) \, \gamma(z) \, (1 + |\zeta|), \quad 
  | \eta(x, \zeta;z)-\eta(y, \xi;z)|  & \leq \gamma(z) \big(\lambda^* |\zeta-\xi| + L_\eta |x-y|\big)\,.
 \end{align*}
 We denote 
 \begin{align*}
     c_{\gamma}:= \int_{\tt E} (\gamma^2(z))\,m({\rm d}z) < + \infty.
 \end{align*}
\end{Assumptions}
\subsection{Notion of solutions and main theorem} First we define notion of strong solution for \eqref{eq:nonlinear}.
 
 \begin{defi}\label{defi:strong-solun}(Strong solution) 
 Let  $(\Omega, \mathcal{F}, \mathbb{P}, \{\mathcal{F}_t\}_{t\ge 0} )$ be a given complete stochastic basis, $W$ be a cylindrical Wiener process in $L^2$ and $N$ be a  time-homogeneous Poisson random measure on ${\tt E}$ with intensity measure $m({\rm d}z)$ defined on $(\Omega, \mathcal{F}, \mathbb{P}, \{\mathcal{F}_t\}_{t\ge 0} )$.  Let $U$ be a $L^2$-valued $\{\mathcal{F}_t\}$-predictable stochastic process with  $ U \in L^2(\Omega; L^2(0,T;L^2))$ and $u_0 \in L^2$ be given. We say that a predictable process $u:\Omega \times[0,T]\goto L^2$ is a strong solution of  \eqref{eq:nonlinear}, if and only if $ u\in L^2(\Omega; \mathbb{D}([0,T];L_w^2))\cap L^p(\Omega;L^p(0,T; W_0^{1,p}))$ such that $u(0,\cdot)=u_0$ in $L^2$ and 
$\mathbb{P}$-a.s., and for all $t\in [0,T]$, 
   \begin{align}
   u(t)& = u_0 +  \int_0^t U(s)\,{\rm d}s + \int_{0}^t \mbox{div}_x \big({\tt A}(x, u(s),\nabla u(s)) + \vec{F}(u(s))\big)\,{\rm d}s \notag \\
  & \hspace{1cm}+ \int_0^t {\tt h}(u(s))\,dW(s) +  \int_0^t \int_{\tt E}\eta(x,u(s);z)\widetilde{N}({\rm d}z,{\rm d}s)  \quad \text{in}~~ W^{-1,p^\prime}\,. \label{eq:defi-weak-form}
  \end{align}
  \end{defi}
In general, for a non-Lipschitz drift and diffusion operators, it may not be possible to prove existence of a strong solution.
  \begin{defi}
  We say that a $8$-tuple $\bar{\pi}=\big( \bar{\Omega}, \bar{\mathcal{F}}, \bar{\mathbb{P}}, \{\bar{\mathcal{F}}_t\},  
\bar{W}, \bar{N}, \bar{u}, \bar{U}\big)$ is a weak solution of \eqref{eq:nonlinear}, if
 \begin{itemize}
  \item [(i)] $(\bar{\Omega}, \bar{\mathcal{F}},\bar{\mathbb{P}}, \{\bar{\mathcal{F}}_t\}_{t\ge 0} )$ is a complete stochastic basis,
  \item[ii)] $\bar{W}$ is a cylindrical Wiener noise on $L^2$, and  $\bar{N}$ is a time-homogeneous Poisson random measure on ${\tt E}$ with intensity measure $m({\rm d}z)$ defined on  $(\Bar{\Omega}, \Bar{\mathcal{F}}, \Bar{\mathbb{P}}, \{\Bar{\mathcal{F}}_t\}_{t\ge 0} )$,
  \item[(iii)] $\bar{U}$ is  a $L^2$-valued $\{\bar{\mathcal{F}_t}\}$-predictable stochastic process with  $ \bar{U}\in L^2(\bar{\Omega}; L^2(0,T;L^2))$
   \item[(iv)] $\bar{u}$ is a $L^2$-valued $\{\bar{\mathcal{F}}_t\}$-predictable  process with $ \bar{u}\in L^2(\bar{\Omega}; \mathbb{D}([0,T];L_w^2))\cap L^p(\bar{\Omega};L^p(0,T; W_0^{1,p}))$ such that 
$\bar{\mathbb{P}}$-a.s., and for all $t\in [0,T]$ \eqref{eq:defi-weak-form} holds. 
 \end{itemize}
 \end{defi}

  \begin{rem} 
 A-priori, it is not known that  $ \int_{0}^t \mbox{div}_x \big({\tt A}(x, u(s),\nabla u(s)) + \vec{F}(u(s))\big)\,{\rm d}s \in L^2$. Since the rest of the terms in  \eqref{eq:defi-weak-form} are in $L^2$ for all $t\in [0,T]$, the equality \eqref{eq:defi-weak-form} holds also in $L^2$. 
 \end{rem}
\begin{thm}[Existence of weak solution]\label{thm:existence-weak}
 Let $W$ and  $N$ be a $L^2$-valued cylindrical Wiener process and time-homogeneous Poisson random measure on ${\tt E}$ with the intensity measure $m({\rm d}z)$ respectively defined on the given stochastic basis   $\big(\Omega, \mathcal{F}, \mathbb{P}, \{\mathcal{F}_t\} \big)$ and $U\in L^2(\Omega; L^2(0,T;L^2))$. Then, under the assumptions  \ref{A1}-\ref{A6}, there exists a weak solution
$\bar{\pi}=\big( \bar{\Omega}, \bar{\mathcal{F}},
\bar{\mathbb{P}},\{\bar{\mathcal{F}}_t\}, \bar{W},
 \bar{N}, \bar{u}, \bar{U}\big)$ of \eqref{eq:nonlinear} such that  $\mathcal{L}(U)=\mathcal{L}(\bar{U})$ on $L^2(0,T;L^2)$. Moreover, there exists $C>0$ such that
\begin{equation}\label{esti:bound-weak-solun}
 \bar{\mathbb{E}}\Big[\sup_{0\le t\le T} \|\bar{u}(t)\|_{L^2}^2 + \int_0^T \| \bar{u}(t)\|_{W_0^{1,p}}^p\,{\rm d}t \Big] \le C\Big( \|u_0\|_{L^2}^2 + \bar{\mathbb{E}}\Big[\int_0^T \|\bar{U}(t)\|_{L^2}^2\,{\rm d}t\Big]\Big)\,.
\end{equation}
\end{thm}

\begin{thm}[Path-wise uniqueness]\label{thm:pathwise-uniqueness}
 Let the assumptions \ref{A1}-\ref{A6} hold. If $u_1$ and $u_2$ are two  martingale solutions of  \eqref{eq:nonlinear} with a control $U$ defined on the same stochastic basis $\big(\Omega,  \mathcal{F},\mathbb{P}, \mathbb{F}:=\{\mathcal{F}_t\}\big)$, then $\mathbb{P}$-a.s., $u_1(t,x)=u_2(t,x)$ for a.e. $(t,x)\in D_T$. 
 \end{thm}
 We use Theorems \ref{thm:existence-weak}  and \ref{thm:pathwise-uniqueness} to prove existence of a unique strong solution of \eqref{eq:nonlinear}.
 \begin{thm}[Existence of strong solution]\label{thm:existence-strong}
 Under the assumptions \ref{A1}-\ref{A6}, there exists a path-wise unique strong solution of  \eqref{eq:nonlinear}. 
 \end{thm}
 \begin{rem}
 Though in view of Definition \ref{defi:strong-solun}, the strong solution (probabilistically strong) has same regularity as of weak solution of \eqref{eq:nonlinear}, one can view from \cite[Theorem $4.2.5$]{claud}) that, in case of Brownian noise only, $u$ is more regular. In particular 
 $u\in L^2(\Omega; C([0,T];L^2))\cap L^p(\Omega;L^p(0,T; W_0^{1,p}))$---this regularity is also reported in \cite{Vallet2021}. 
 \end{rem}
 By $\mathcal{U}_{\rm ad}^w(u_0;T)$, we denote the set of weak admissible solutions to \eqref{eq:nonlinear}.
\begin{defi}\label{defi:optimal-control}
Let ${\tt \Psi}$ be a given Lipschitz continuous function on $L^2$, ${\tt u}_{\rm det} \in L^p(0,T; W_0^{1,p})$ and the assumptions of Theorem \ref{thm:existence-weak} holds true. We say that
 $\pi^*:=\big( \Omega^*, \mathcal{F}^*, \mathbb{P}^*,\{ \mathcal{F}_t^*\}, W^*, N^*, u^*, U^*\big)\in \mathcal{U}_{\rm ad}^w(u_0;T)$ is  a weak optimal solution of \eqref{eq:control-problem} if 
 \begin{align*}
  \mathcal{\pmb J}(\pi^*)= \inf_{\pi \in \mathcal{U}_{\rm ad}^w(u_0; T)} \mathcal{\pmb J}(\pi):=\Lambda\,.
 \end{align*}
\end{defi}

\begin{thm}\label{thm:existence-optimal-control}
 There exists a weak optimal solution $\pi^*$ of the control problem \eqref{eq:control-problem}.
\end{thm}

\subsection{Invariant measure}
Consider the uncontrolled SPDEs 
\begin{align} 
 {\rm d} u -\mbox{div}_x \big( {\tt A}(x, u, \grad u) + \vec{F}(u)\big)\,{\rm d}t =  {\tt h}(u)\,dW(t) + \int_{\tt E}\eta(x,u;z)\widetilde{N}({\rm d}z,{\rm d}t) ; \quad u(0,x)=u_0\,. \label{eq:nonlinear-no-control}
\end{align} 
i.e., \eqref{eq:nonlinear} with $U=0$.
Let $u(t,v)$ denote the path-wise unique strong solution of \eqref{eq:nonlinear-no-control}~(which exists due to Theorem  \ref{thm:existence-strong}) with fixed initial data $u_0=v\in L^2$. Define a family of Markov semi-group $\{P_t\}_{t\ge 0}$ as
 \begin{align}
(P_t \phi)(v):=\mathbb{E}\big[\phi(u(t,v))\big], \quad \phi \in\mathcal{B}(L^2)\,, \label{defi:semi-group}
\end{align}
 where $\mathcal{B}(L^2)$ denotes the space of all bounded Borel measurable functions on $L^2$. 
 
Now, we state the following main theorem regarding the existence of invariant measure.
\begin{thm}\label{thm:existence-invariant-measure}
Let the assumptions \ref{A1}-\ref{A6} hold true. In addition assume that there exists $\delta>0$ such that 
\begin{align}
C_{hg} + 2C_1 \|\nabla u\|_{L^p}^p-\|{\tt h}(u)\|_{\mathcal{L}_2(L^2)}^2 - \int_{\tt E} \|\eta(\cdot ,u;z) \|_{L^2}^2 m({\rm d}z) \ge \delta \|u\|_{L^2}^p\,,\label{cond:extra-sigma}
\end{align}
where $C_{hg} = 2 (C_5+ c_{\gamma} \|g\|_{L^\infty}^{2}) (1+|D|)$.
Then there exists an invariant measure $\mu \in \mathcal{P}(L^2)$, the set of all Borel probability measures on $L^2$, for $\{P_t\}$ i.e., 
$P_t^* \mu=\mu$, where $P_t^*$ denotes the adjoint semi-group acting on $\mathcal{P}(L^2)$ given  by 
$$P_{t}^{*} \nu (\Gamma) = \int_{L^2} {\tt P}_{t}(x, \Gamma)\, \nu({\rm d}x) \quad \text{with}~ {\tt P}_t(v, \Gamma) := \mathbb{P} \big(u(t,v) \in \Gamma \big)~\text{for any Borel set} \ \Gamma \subset L^2 .$$
In particular, for any $t\ge 0$, there holds
\begin{align*}
\int_{L^2} P_t\phi \,d\mu=\int_{L^2} \phi\,d\mu, \quad \forall ~\phi \in SC_b(L_w^2)\,,
\end{align*}
where  $SC_b(L_w^2)$ denotes the space of bounded, sequentially weakly continuous function on $L^2$; see Definition \ref{defi:se-weak-bounded}. 
\end{thm}
\section{Well-posedness of solution}\label{sec:existence-weak-solu}
In this section, we wish to show  well-posedness of a strong solution for \eqref{eq:nonlinear}. We first show existence of a martingale solution and then prove its path-wise uniqueness result.
\subsection{Existence of weak solution:}
In this subsection, we prove existence of a martingale solution for \eqref{eq:nonlinear}. We use semi-implicit time discretization scheme to construct  an approximate solutions and then derive its necessary a-priori bounds.

\subsubsection{\bf Discretization and projection in time.} For fixed $T>0$, consider a uniform partition $\{0=t_0< t_1< \cdots < t_N=T\} $ of $[0,T]$ with mesh size $\kappa:=\frac{T}{N}, N\in \mathbb{N}^*$. By  \cite[Lemma $30$]{Vallet2019}, there exists $\{ u_{0,\kappa}\}_{\kappa>0}\subset W_0^{1,p}$ such that $u_{0,\kappa} \goto u_0$ in $L^2$ as $\kappa \goto 0$ and 
\begin{align}
 \frac{1}{2}\|u_{0,\kappa} \|_{L^2}^2 + \kappa \, \|\nabla u_{0,\kappa} \|_{L^p}^p \le \frac{1}{2}\|u_0\|_{L^2}^2\,. \label{esti:initial-approx}
\end{align}
The evaluation of the control $U$ at time $t_i$ does not make sense. Therefore, it is require to define a projection operator in time. Define the set 
\begin{align*}
 \mathcal{P}_\kappa:= \Big\{ \phi_\kappa: (0,T)\goto L^2: \,\, \phi_\kappa \big|_{(t_j, t_{j+1}]} \,\,\text{is a constant in}\, L^2\Big\}. 
\end{align*}
For any $f\in L^2(0,T;L^2)$, we define the projection $\Pi_\kappa f \in \mathcal{P}_\kappa$ via 
\begin{align*}
 \int_0^T \big( \Pi_\kappa f-f, \phi_\kappa \big)_{L^2}\,{\rm d}t = 0 \quad \forall \phi_\kappa \in \mathcal{P}_\kappa.
\end{align*}
Then \cite[Lemma $4.4$]{Xin} gives
\begin{align}\label{eq:projection_f+projection_convergence}
\|\Pi_\kappa f\|_{L^2(0,T;L^2)} \le \|f\|_{L^2(0,T;L^2)}, \quad 
 \Pi_\kappa f \rightarrow f \quad \text{in}\,\,\, L^2(0,T;L^2).
 \end{align}
For $0\le i\le N$, set $U_i=\Pi_{\kappa}U (t_i)$. 
By using \eqref{eq:projection_f+projection_convergence}, we have
\begin{align}
  \mathbb{E}\Big[  \kappa \sum_{i=0}^{j-1} \|U_{i+1}\|_{L^2}^2\Big] \le\mathbb{E}\Big[\int_0^T \|\Pi_{\kappa} U\|_{L^2}^2\,{\rm d}s\Big]  \le
  \mathbb{E}\Big[\int_0^T \|U\|_{L^2}^2\,{\rm d}s\Big]\le  C\,. \label{inq:projection-control}
  \end{align}   
  \subsubsection{\bf Semi-discrete scheme and its solvability} 
We consider the following time discretization scheme: for $\hat{u}_0= u_{0,\kappa}$, find $\mathcal{F}_{t_{k+1}}$-measurable $W_0^{1,p}$-valued random variable
 $\hat{u}_{k+1},~k=0,1,\ldots, N-1$ such that $\mathbb{P}$-a.s., 
  \begin{align}
  &  \hat{u}_{k+1}- \hat{u}_k -\kappa\, \mbox{div}_x\big\{ {\tt A}(x, \hat{u}_{k+1},\nabla \hat{u}_{k+1}) + \vec{F}(\hat{u}_{k+1}) \big\} \notag \\
  & \hspace{2cm} = \kappa U_k + {\tt h}(\hat{u}_k)\Delta_k W +  \int_{\tt E}\int_{t_k}^{t_{k+1}}\eta(x,\hat{u}_k;z) \widetilde{N}({\rm d}z,{\rm d}t)\,,\label{eq:discrete-nonlinear}
  \end{align}
  holds in  $W_0^{-1,p^\prime}$, where $\Delta_k W=W(t_{k+1})-W(t_k)$. 
  Solvability of \eqref{eq:discrete-nonlinear} is guaranteed in the following proposition.
  \begin{prop}
   There exists a unique $W_0^{1,p}$-valued $\mathcal{F}_{t_{k+1}}$-measurable random variable
 $\hat{u}_{k+1},~k=0,1,\ldots, N-1$ satisfying the variational formula: for $\mathbb{P}$-a.s., 
  \begin{align}
    &\int_{D} \Big( \big\{\hat{u}_{k+1}- \hat{u}_k\big\}v + \kappa\big({\tt A}(x, \hat{u}_{k+1},\nabla \hat{u}_{k+1}) + \vec{F}(\hat{u}_{k+1})\big)\cdot \grad v \Big)\,{\rm d}x -\kappa \int_D U_k v\,{\rm d}x \notag \\
    &= \int_D \Big(\int_{t_k}^{t_{k+1}} {\tt h}(\hat{u}_k)\, dW(s)\Big) v\,{\rm d}x +  \int_{D} \int_{t_k}^{t_{k+1}} \int_{\tt E} \eta(x,\hat{u}_k;z)\,v\, \widetilde{N}({\rm d}z,{\rm d}s)\,{\rm d}x \quad \forall\,v\in W_0^{1,p}\,. \label{variational_formula_discrete}
   \end{align}
  \end{prop}
  \begin{proof}
  We follow the idea of \cite{Vallet2021}. Define an operator $\mathcal{A}: W_0^{1,p}\goto W^{-1,p^\prime}$ via
   \begin{align*}
  \big\langle \mathcal{A}u, v \big\rangle := \int_{D}\Big( uv + \kappa {\tt A}(x, u,\nabla u)\cdot \nabla v \Big)\,{\rm d}x
  \quad  \forall\,u,v \in W_0^{1,p}\,.
  \end{align*}
  A similar argument as done in \cite[pp. $269$]{Vallet2021} reveals that $\mathcal{A}$ is a pseudomonotone, coercive, and bounded operator and hence by Brezis' theorem $\mathcal{A}$ is onto $W^{-1,p^\prime}(D)$, see \cite[Theorem $2.6$]{Roubicek}. Moreover, $\mathcal{A}$ is injective and  $\mathcal{A}^{-1}:W^{-1,p^\prime} \goto W_0^{1,p}$ is demi-continuous; see \cite[Lemma $2.1$]{Vallet2021}. Since $W_0^{1,p}$ is separable, by Dunford-Pettis theorem,
  $\mathcal{A}^{-1}:W^{-1,p^\prime} \goto W_0^{1,p}$ is continuous. Setting
 $$ \quad  \displaystyle X_k:= \hat{u}_k + \kappa U_k + \int_{t_k}^{t_{k+1}}{\tt h}(\hat{u}_k)\,dW(s)+ \int_{\tt E}\int_{t_k}^{t_{k+1}}\eta(x,\hat{u}_k;z) \widetilde{N}({\rm d}z,{\rm d}t),$$
 and applying It\^{o}-L\'{e}vy isometry together with the assumptions \ref{A4},
\ref{A6}, we get 
\begin{align*}
 \mathbb{E}\big[ \|X_k\|_{L^2}^2\big] & \le 4\mathbb{E}\Big\{ \|\hat{u}_k\|_{L^2}^2 + \kappa^2 \|U_k\|_{L^2}^2 +  \kappa C_5(|D| +\|\hat{u}_k\|_{L^2}^2)+ 2 \kappa\, \|g\|_{L^\infty}^{2}  \big( |D| + \|\hat{u}_{k}\|_{L^2}^2\big)c_\gamma
 \Big\}\\
 & \le C(\kappa, |D|, C_5,\|g\|_{L^\infty}, U) \Big( 1 + \mathbb{E}\big[\|\hat{u}_k\|_{L^2}^2\big]\Big).
\end{align*}
Therefore $\mathbb{P}$-a.s., $X_k \in L^2$, and hence we have $\hat{u}_{k+1}= \mathcal{A}^{-1}X_k$. Thanks to continuity of $\mathcal{A}^{-1}:W^{-1,p^\prime} \goto W_0^{1,p}$, and mathematical induction, we arrive at the required assertion.
\end{proof}
\subsubsection{\bf {\it A-priori} estimates}
We derive necessary uniform bounds for $\hat{u}_{k+1}$ (uniform in $k$ and $\kappa$). 
\begin{lem}\label{lem:a-priori:1}
Let the assumptions \ref{A1}-\ref{A6} hold true, and $\hat{u}_{k+1}$ be a solution of \eqref{eq:discrete-nonlinear} for $0\le k\le N-1$.
Then there exists a constant $C>0$ such that
\begin{align}
 \sup_{1\le n\le N}\mathbb{E}\Big[\| \hat{u}_{n}\|_{L^2}^2\Big] +  \sum_{k=0}^{N-1} \mathbb{E}\Big[\| \hat{u}_{k+1}-\hat{u}_k\|_{L^2}^2\Big] +
  \kappa \sum_{k=0}^{N-1} \mathbb{E}\Big[\|\grad \hat{u}_{k+1}\|_{L^p}^p\Big]  \le  C\,. \label{a-prioriestimate:1}
\end{align}
\end{lem}
\begin{proof}
We follow the ideas from \cite{Majee2020,Vallet2019,Vallet2021}. 
Taking $v=\hat{u}_{k+1}$ as a test function in \eqref{variational_formula_discrete}, and using Young's inequality, and the identity
\begin{align}
 ({\tt a}-{\tt b}){\tt a}= \frac{1}{2}\big[{\tt a}^2 + ({\tt a}-{\tt b})^2 -{\tt b}^2\big] \quad \forall\,{\tt a},{\tt b} \in \R, \label{eq:identity-0}
\end{align}
we get
\begin{align}
 &\frac{1}{2} \Big\{\| \hat{u}_{k+1}\|_{L^2}^2+ \| \hat{u}_{k+1}-\hat{u}_k||_{L^2}^2 - \|\hat{u}_k\|_{L^2}^2 \Big\}  
  + \kappa\,   \int_D \big({\tt A}(x, \hat{u}_{k+1},\nabla \hat{u}_{k+1}) + F(\hat{u}_{k+1})\big) \cdot \nabla \hat{u}_{k+1}\,{\rm d}x  \notag \\
 &  \le \frac{\kappa}{2}\|U_k\|_{L^2}^2 + \frac{\kappa}{2} \|\hat{u}_{k+1}\|_{L^2}^2 + 
  \int_{D} \Big(\int_{t_k}^{t_{k+1}} {\tt h}(\hat{u}_k)\, dW(s)\Big) \hat{u}_k\,{\rm d}x
  + \frac{1}{2\eps} \int_{D} \Big(\int_{t_k}^{t_{k+1}} {\tt h}(\hat{u}_k)\, dW(s)\Big)^2\,{\rm d}x\notag \\
& \quad + \int_{D} \int_{t_k}^{t_{k+1}} \int_{\tt E} \eta(x,\hat{u}_k;z) \hat{u}_k\, \widetilde{N}({\rm d}z,{\rm d}s)\,{\rm d}x
 + \frac{\eps}{4} \| \hat{u}_{k+1}- \hat{u}_k\|_{L^2}^2 \notag \\
  & \qquad  + \frac{1}{2\eps} \int_{D} \Big(\int_{t_k}^{t_{k+1}} \int_{\tt E} \eta(x,\hat{u}_k;z)\, \widetilde{N}({\rm d}z,{\rm d}s)\Big)^2\,{\rm d}x\,,\quad
  \text{for some}~~\eps >0 \,. \label{esti:1}
  \end{align}
  Observe that, in view of Gauss-Green's theorem and Assumption \ref{A2}, {\rm ii)},
\begin{align}
 \kappa \int_D \Big({\tt A}(x, \hat{u}_{k+1},\nabla \hat{u}_{k+1}) + F(\hat{u}_{k+1})\Big) \cdot \nabla \hat{u}_{k+1}\,{\rm d}x 
 \ge \kappa C_1 \|\nabla \hat{u}_{k+1}\|_{L^p}^p - \kappa \|K_1\|_{L^1}\,. \label{esti:2}
\end{align}
Combining \eqref{esti:1} and \eqref{esti:2} and using It\^{o}-L\'{e}vy isometry along with the assumptions \ref{A3}-\ref{A6}, we have, for $\eps=1$, 
\begin{align}
 &\frac{1}{2}  \mathbb{E}\big[\| \hat{u}_{k+1}\|_{L^2}^2- \|\hat{u}_k\|_{L^2}^2\big] 
 + \frac{1}{4}  \mathbb{E}\big[\| \hat{u}_{k+1}-\hat{u}_k\|_{L^2}^2 \big]  + \kappa C_1 \mathbb{E}\big[ \|\nabla \hat{u}_{k+1}\|_{L^p}^p\big] \notag \\
 & \le \kappa \|K_1\|_{L^1} +\kappa \mathbb{E}\big[\|U_k\|_{L^2}^2\big] + \frac{\kappa}{4} \mathbb{E}[\|\hat{u}_{k+1}\|_{L^2}^2] + \frac{1}{2}\mathbb{E}\Big[\int_{t_k}^{t_{k+1}} \|{\tt h}(\hat{u}_k)\|_{\mathcal{L}_2(L^2)}^2\,{\rm d}s\Big] \notag  \\
 & \hspace{2cm} +  \frac{1}{2}\mathbb{E}\Big[  \int_{D} \int_{t_k}^{t_{k+1}} \int_{\tt E} \eta^2(x, \hat{u}_k;z)\, m({\rm d}z)\,{\rm d}s\,{\rm d}x\Big] \notag \\
 & \le \kappa \Big(\|K_1\|_{L^1} + C_5 |D|+ \|g\|_{L^2}^2 \, c_{\gamma} \Big) + \kappa \mathbb{E}[\|U_k\|_{L^2}^2] + \frac{\kappa}{4} \mathbb{E}[\|\hat{u}_{k+1}\|_{L^2}^2]  \notag \\
 & \hspace{2cm}+
 \kappa \big(C_5+ \|g\|_{L^\infty}^2 \, c_{\gamma} \big)\mathbb{E}[\|\hat{u}_k\|_{L^2}^2]\,. \label{inq:apriori-1}
\end{align}
 Summing over all time steps $0\le k\le n-1$ \eqref {inq:apriori-1} and using the fact that 
 $\kappa \le 1$, we have 
 \begin{align*}
\frac{1}{2}\big(\mathbb{E}[\|\hat{u}_{n}\|_{L^2}^2] - \|\hat{u}_0\|_{L^2}^2\big) 
\le C + \kappa \sum_{k=0}^{n-1}\mathbb{E}[\|U_k\|_{L^2}^2] + \frac{1}{4} \mathbb{E}[\|\hat{u}_{n}\|_{L^2}^2] + \kappa C \sum_{k=0}^{n-1} \mathbb{E}[\|\hat{u}_{k}\|_{L^2}^2]\,.
 \end{align*}
One can arrive at the assertion \eqref{a-prioriestimate:1} by using  discrete Gronwall's lemma 
 together with \eqref{inq:projection-control}. 
\end{proof}

\subsubsection{ \bf Continuation and its bound:}
Using the discrete solutions $\{ \hat{u}_k\}$, we first define stochastic processes
which are defined on the time interval $[0,T]$ and then deduce uniform bounds.
Like in \cite{Majee2020,Vallet2019}, we
introduce the following processes:
\begin{equation*}
\begin{aligned}
 u_{\kappa}(t):&= \sum_{k=0}^{N-1} \hat{u}_{k+1} {\bf 1}_{[t_k,t_{k+1})}(t),~~~
 \bar{u}_{\kappa}(t):= \sum_{k=0}^{N-1} \hat{u}_{k} {\bf 1}_{(t_k,t_{k+1}]}(t), ~~t\in (0,T],
 u_\kappa (T)=\hat{u}_N,~~ \\
 U_{\kappa}(t):&= \sum_{k=0}^{N-1} U_{k+1} {\bf 1}_{[t_k,t_{k+1})}(t),~~ t\in [0,T), \quad U_\kappa (T)=U_N,~~~\bar{u}_{\kappa}(0)=\hat{u}_0\,, \\
 M_{\kappa}(t):&= \int_0^t {\tt h}(\bar{u}_\kappa(s))\,dW(s),\quad 
 J_{\kappa}(t):= \int_0^t \int_{\tt E} \eta(x, \bar{u}_{\kappa}(s);z)\widetilde{N}({\rm d}z,{\rm d}s),~~t\in [0,T]\,,
 \end{aligned}
\end{equation*}
and for $t\in [0,T)$,
\begin{equation}\label{eq:defi-piecewise-affine-function}
 \begin{aligned}
  \tilde{u}_{\kappa}(t):&= \sum_{k=0}^{N-1}\Big( \frac{\hat{u}_{k+1}-\hat{u}_k}{\kappa}(t-t_k) + \hat{u}_k\Big) {\bf 1}_{[t_k,t_{k+1})}(t), \quad 
    \tilde{u}_{\kappa}(T):= \hat{u}_N\,, \\
    \tilde{J}_{\kappa}(t):&= \sum_{k=0}^{N-1}\Big( \frac{J_{\kappa}(t_{k+1})-J_{\kappa}(t_k)}{\kappa}(t-t_k)
   + J_{\kappa}(t_k)\Big) {\bf 1}_{[t_k,t_{k+1})}(t),~~\tilde{J}_{\kappa}(T):= J_\kappa(T)\,, \\
      \tilde{M}_{\kappa}(t):&= \sum_{k=0}^{N-1}\Big( \frac{ M_{\kappa}(t_{k+1})-M_{\kappa}(t_k)}{\kappa}(t-t_k) + M_{\kappa}( t_k)\Big) {\bf 1}_{[t_k,t_{k+1})}(t),~~\tilde{M}_{\kappa}(T):= M_{\kappa}(T)\,.
 \end{aligned}
\end{equation}
In the following lemma, we derive {\em a-priori} estimates for $\tilde{u}_\kappa$ and $u_{\kappa}$.
\begin{lem}\label{lem:a-priori:2}
Under the assumptions \ref{A1}-\ref{A6}, there exists a constant $C>0$ (independent of $\kappa$) such that
\begin{itemize}
  \item [(i)] $   \displaystyle \underset{t\in [0,T]}\sup\, \mathbb{E}\big[\| \tilde{u}_{\kappa}(t)\|_{L^2}^2\big] \le C,\quad 
  \mathbb{E}\Big[ \int_0^T \|u_{\kappa}(t)-\tilde{u}_{\kappa}(t)\|_{L^2}^2\,{\rm d}t \Big] \le C \kappa$, 
\item[ii)] $  \displaystyle \mathbb{E}\Big[\int_0^T\|\nabla u_{\kappa}(t)\|_{L^p}^p\,{\rm d}t \Big] \le C, \quad 
  \mathbb{E}\Big[\int_0^T \|\nabla \tilde{u}_{\kappa}(t)\|_{L^p}^p\,{\rm d}t \Big] \le C$, 
 \item[(iii)] 
 $ \mathbb{E}\Big[  \underset{t\in [0,T]}\sup\,\| \tilde{u}_{\kappa}(t)\|_{L^2}^2\Big] \le C, \quad 
   \mathbb{E}\Big[ \big\|\tilde{u}_{\kappa}\big\|_{L^p(0,T; W_0^{1,p})}^p \Big] \le C$,
\end{itemize}
\end{lem}
\begin{proof}
A simple calculation reveals that 
\begin{align*}
\begin{cases}
  \displaystyle   \underset{t\in [0,T]}\sup\, \mathbb{E}\big[\| \tilde{u}_{\kappa}(t)\|_{L^2}^2\big] = \underset{t\in [0,T]}\sup\, \mathbb{E}\big[\| u_{\kappa}(t)\|_{L^2}^2\big] = \underset{0\le k\le N-1}\max\, \mathbb{E}\big[\|\hat{u}_{k+1}\|_{L^2}^2\big] \,, \\
   \displaystyle    \mathbb{E}\Big[ \int_0^T \|u_{\kappa}(t)-\tilde{u}_{\kappa}(t)\|_{L^2}^2\,{\rm d}t \Big] \le \kappa  \sum_{k=0}^{N-1} \mathbb{E}\Big[\| \hat{u}_{k+1}-\hat{u}_k\|_{L^2}^2\Big]\,, \\
  \displaystyle  \mathbb{E}\Big[\int_0^T\|\nabla u_{\kappa}(t)\|_{L^p}^p\,{\rm d}t \Big] \le \kappa \sum_{k=0}^{N-1} 
 \mathbb{E}\big[ \|\nabla \hat{u}_{k+1}\|_{L^p}^p \big] \, , \\
  \displaystyle  \mathbb{E}\Big[\int_0^T\|\nabla \tilde{u}_{\kappa}(t)\|_{L^p}^p\,{\rm d}t \Big] \le \frac{\kappa}{2} \sum_{k=0}^{N-1} \big( \mathbb{E}[\|\nabla \hat{u}_{k+1}\|_{L^p}^p ] +  \mathbb{E}[\|\nabla \hat{u}_{k}\|_{L^p}^p ] \big)\,.
 \end{cases}
\end{align*}
Hence the assertions ${\rm (i)}$ and  ${\rm (ii)}$ follows from \eqref{a-prioriestimate:1} and \eqref{esti:initial-approx}\,.
\vspace{.2cm}

\noindent{Proof of ${\rm (iii)}$.} We prove the first part of ${\rm (iii)}$.
We combine \eqref{esti:1} and \eqref{esti:2} for $\eps=1$, and then sum over $k=0,1,\ldots, n-1$ and take the maximum over 
$n=1,2,\ldots, N$. The resulting expression yields, after taking expectation,
\begin{align*}
& \mathbb{E}\Big[ \max_{1\le n \le N} \|\hat{u}_n\|_{L^2}^2\Big] \\
&\le \mathbb{E}\big[ \|\hat{u}_0\|_{L^2}^2\big] + 2\,\kappa\mathbb{E}\Big[ \sum_{k=0}^{N}\|U_k\|_{L^2}^2 \Big] + \kappa \sum_{k=0}^{N-1} \mathbb{E}\big[ \|\hat{u}_{k+1}\|_{L^2}^2\big] + 2T
\|K_1\|_{L^1} + \mathbb{E}\Big[ \sum_{k=0}^{N-1} \big\| \int_{t_k}^{t_{k+1}} {\tt h}(\hat{u}_k)\,dW(s)\big\|_{L^2}^2\Big] \notag \\
&  + 2\, \mathbb{E}\Big[  \max_{1 \le n \le N} \sum_{k=0}^{n-1}  \int_{D} \int_{t_k}^{t_{k+1}}  {\tt h}(\hat{u}_k) \hat{u}_k\, dW(s)\,{\rm d}x\Big] 
 + \mathbb{E}\Big[ \sum_{k=0}^{N-1} \big\| \int_{t_k}^{t_{k+1}}\int_{\tt E} \eta(\cdot,\hat{u}_k;z) \widetilde{N}({\rm d}z,{\rm d}s)\big\|_{L^2}^2\Big] \\
&  \qquad  + 2 \mathbb{E}\Big[ \max_{1 \le n \le N} \sum_{k=0}^{n-1}  \int_{D} \int_{t_k}^{t_{k+1}} \int_{\tt E} \eta(x,\hat{u}_k;z) \hat{u}_k\, \widetilde{N}({\rm d}z,{\rm d}s)\,{\rm d}x\Big] \\
&:= \mathbb{E} \big[\|\hat{u}_0\|_{L^2}^2\big] +2\, \kappa\mathbb{E}\Big[ \sum_{k=0}^{N}\|U_k\|_{L^2}^2 \Big] + \kappa \sum_{k=0}^{N-1} \mathbb{E}\big[ \|\hat{u}_{k+1}\|_{L^2}^2\big] + 2T
\|K_1\|_{L^1} +\sum_{i=1}^4 {\bf  I}_i\,.
\end{align*}
Using It\^{o}-L\'{e}vy isometry, and the assumption \ref{A6}, we obtain that
\begin{align*}
{\bf I}_3 =  \mathbb{E}\Big[ \int_0^T \int_{\tt E} \|\eta(\cdot, \bar{u}_\kappa(s);z)\|_{L^2}^2\,m({\rm d}z)\,{\rm d}s\Big]
&\le 2c_{\gamma} \Big( \|g\|_{L^2}^2 T + \|g\|_{L^\infty}^2 \mathbb{E}\Big[ \int_0^T \|\bar{u}_\kappa(s)\|_{L^2}^2\,{\rm d}s\Big]\Big) \notag \\
& \le 2c_{\gamma} \Big( \|g\|_{L^2}^2 T + \|g\|_{L^\infty}^2 \kappa \sum_{k=0}^N \mathbb{E}\big[ \|\hat{u}_k\|_{L^2}^2\big] \Big) \le C\,,
\end{align*}
where the last inequality follows from  \eqref{a-prioriestimate:1} and \eqref{esti:initial-approx}.  Similarly, by using It\^{o}- isometry, Assumption \ref{A4} and   \eqref{a-prioriestimate:1}, one has 
\begin{align*}
{\bf I}_1 &=  \mathbb{E}\Big[ \int_0^T  \|{\tt h}(\bar{u}_\kappa(s))\|_{\mathcal{L}_2(L^2)}^2\,{\rm d}s\Big]
\le 2 \Big( TC_5|D| + C_5\mathbb{E}\Big[ \int_0^T \|\bar{u}_\kappa(s)\|_{L^2}^2\,{\rm d}s\Big]\Big)\le C\,.
\end{align*}
Observe that, by using BDG inequality along with Young's inequality, we get
\begin{align*}
{\bf I}_2 + {\bf I}_4 & \le 2   \mathbb{E}\Big[ \max_{1 \le n \le N}  \Big|\int_{0}^{t_{n}} \int_{\tt E}  \big( \eta(x,\bar{u}_\kappa(s);z), \bar{u}_\kappa(s)\big)_{L^2} \widetilde{N}({\rm d}z,{\rm d}s) \Big|\Big] \notag \\
& \quad +  2   \mathbb{E}\Big[ \max_{1 \le n \le N}  \Big|  \big( \int_{0}^{t_{n}}{\tt h}(\bar{u}_\kappa(s))\,dW(s), \bar{u}_\kappa(s)\big)_{L^2}\Big|\Big] \notag \\
 & \le C  \mathbb{E}\Big[ \Big(\int_{0}^{T} \int_{\tt E}  \big( \eta(\cdot,\bar{u}_\kappa(s);z), \bar{u}_\kappa(s)\big)_{L^2}^2 m({\rm d}z)\,{\rm d}s \Big)^\frac{1}{2}+ \Big(\int_{0}^{T} \| {\tt h}(\bar{u}_\kappa(s)\|_{\mathcal{L}_2(L^2)}^2 \|\bar{u}_\kappa(s)\|_{L^2}^2 \,{\rm d}s \Big)^\frac{1}{2}\Big] \\
& \le C \mathbb{E}\Big[ \sup_{0\le t\le T} \|\bar{u}_\kappa(t)\|_{L^2} \Big\{ \Big( \int_0^T \int_{\tt E} \|\eta(\cdot,\bar{u}_\kappa(s);z)\|_{L^2}^2\,m({\rm d}z)\,{\rm d}s\Big)^\frac{1}{2} + \Big( \int_0^T \|{\tt h}(\bar{u}_\kappa(s))\|_{\mathcal{L}_2(L^2)}^2\,{\rm d}s\Big)^\frac{1}{2}\Big\} \Big] \\
& \le \eps_1 \mathbb{E}\Big[ \sup_{0\le t\le T} \|\bar{u}_\kappa(t)\|_{L^2}^2 \Big] + 
C(\eps_1) \big( {\bf I}_2 + {\bf I}_3\big) \notag \\
& \le \eps_1 \mathbb{E}\Big[ \max_{1\le n \le N} \|\hat{u}_n\|_{L^2}^2\Big] + \eps_1 \mathbb{E}\big[ \|\hat{u}_0\|_{L^2}^2\big]
+ C(\eps_1)  \big( {\bf I}_1 + {\bf I}_3\big)\,.
\end{align*}
Using all these estimates and choosing $\eps_1 <1$, the resulting inequality then takes the form 
\begin{align*}
\mathbb{E}\Big[ \max_{1\le n \le N} \|\hat{u}_n\|_{L^2}^2\Big]  \le C \Big\{  
\|\hat{u}_0\|_{L^2}^2\big]  + \|K_1\|_{L^1} +  \kappa\mathbb{E}\Big[ \sum_{k=0}^{N}\|U_k\|_{L^2}^2 \Big] + \kappa \sum_{k=0}^{N-1} \mathbb{E}\big[ \|\hat{u}_{k+1}\|_{L^2}^2\big] + {\bf I}_2 + {\bf I}_3\Big\}\,,
\end{align*}
for some constant $C$.  The assertion then follows once we use \eqref{a-prioriestimate:1},\eqref{esti:initial-approx} and \eqref{inq:projection-control}\,.
\vspace{.2cm}

 In view of \eqref{esti:initial-approx} and the first part of ${\rm (ii)}$, one can easily see that
\begin{align*}
   \mathbb{E}\Big[ \big\|\tilde{u}_{\kappa}\big\|_{L^p(0,T; W_0^{1,p})}^p \Big] 
 & \le C \kappa \sum_{k=0}^N  \mathbb{E}\big[ \|\hat{u}_k\|_{W_0^{1,p}}^p \big] 
 \le C \mathbb{E}\Big[ \int_0^T \|\grad  u_{\kappa}(t)\|_{L^p}^p\,{\rm d}t + \kappa\,\|\grad  \hat{u}_0\|_{L^p}^p\Big] \notag \\
 & \le C\,\Big(\|u_0\|_{L^2}^2 +  \mathbb{E} \Big[ \int_{0}^T |\|\nabla u_\kappa(t)\|_{L^p}^p\,{\rm d}t\Big] \Big)\le C\,.
\end{align*}
\end{proof}
\subsubsection{\bf Tightness and convergence of approximate solutions}
    Taking motivation from \cite{brzezniakmotyl2013,Majee2020, Metivier1988, Viot1988} and looking at the estimations in Lemma \ref{lem:a-priori:2}, we first show the tightness of $\mathcal{L}(\tilde{u}_{\kappa})$ on the functional space
\begin{align*}
 \mathcal{Z}_T:= \mathbb{D}([0,T]; W^{-1,p^\prime})\cap \mathbb{D}([0,T]; L^2_w)\cap L^2_w(0,T; W^{1,2}) \cap L^2(0,T; L^2)
\end{align*}
  equipped with the supremum topology $\mathcal{T}_T$ ; see \cite{Majee2022}. To proceed further, we recall the definition of Aldous condition. 

\begin{defi}\label{defi:aldous-condition}
(Aldous condition)
Let  $(\mathbb{B},\|\cdot\|_{\mathbb{B}})$ be a separable Banach space. We say that a sequence of $\{\mathcal{F}_t\}$-adapted, $\mathbb{B}$-valued c\`{a}dl\`{a}g processes $({\tt Z}_n)_{n\in \mathbb{N}}$ satisfy the Aldous condition
if for every $\delta,~m>0$, there exists $\eps>0$ such that
\begin{align*}
 \sup_{n\in \mathbb{N}}\sup_{0<\theta\le \eps} \mathbb{P}\big\{\|{\tt Z}_n(\kappa_n + \theta)-{\tt Z}_n(\kappa_n)\|_{\mathbb{B}}\ge m\big\} \le \delta\,
\end{align*} 
holds for every $\{\mathcal{F}_t\}$-adapted sequence of stopping times $(\kappa_n)_{n\in \mathbb{N}}$  with 
$\kappa_n\le T$.
\end{defi}
 We mention a sufficient condition to satisfy Aldous condition; cf.~ \cite[Lemma 9]{Motyl2013}.
\begin{lem}
 The sequence $({\tt Z}_n)_{n\in \mathbb{N}}$ as mentioned in Definition \ref{defi:aldous-condition} be a sequence 
 of c\`{a}dl\`{a}g stochastic processes satisfies the Aldous condition if there exist positive constants $\alpha,\, \zeta$ and $C$ such that for any $\theta >0$,
\begin{align}
 \mathbb{E}\Big[\|{\tt Z}_n(\kappa_n + \theta)-{\tt Z}_n(\kappa_n)\|^\alpha_{\mathbb{B}}\Big] \le C \theta^\zeta \label{inq:Aldous-condition}
\end{align}
holds for every sequence of $\{\mathcal{F}_t\}$-stopping times $(\kappa_n)$ with $\kappa_n\le T$.
\end{lem}
Following the line arguments of \cite[Lemma $3.3$]{brzezniakmotyl2013}, \cite[Lemma $2.5$]{Rozovskii2005}, \cite[Theorem $2$, Lemma $7$,~$\&$~Corollary $1$]{Motyl2013}, we now state~(without proof) tightness criterion of the family of laws of some process in $\mathcal{Z}_T$.
\begin{thm}\label{thm:for-tightness}
 Let $({\tt X}_{\kappa})_{\kappa >0}$ be a sequence of $\mathbb{F}$-adapted, $W^{-1,p^\prime}$-valued c\`{a}dl\`{a}g stochastic process such that 
  \begin{itemize}
  \item[(i)]
   $({\tt X}_{\kappa})_{\kappa >0}$ satisfies the Aldous condition in $W^{-1,p^\prime}$,
 \item[(ii)] for some constant $C>0$, 
  \begin{align*}
   \sup_{\kappa >0} \mathbb{E}\Big[ \sup_{t\in [0,T]}\|{\tt X}_{\kappa}(t)\|_{L^2}\Big] \le C, \quad 
   \sup_{\kappa >0} \mathbb{E}\Big[\int_0^T\|{\tt X}_{\kappa}(t)\|_{W_0^{1,p}}^2\,{\rm d}t\Big] \le C\,.
  \end{align*}
  \end{itemize}
Then the sequence  $\big(\mathcal{L}({\tt X}_{\kappa})\big)_{\kappa >0}$ is tight on $(\mathcal{Z}_T, \mathcal{T}_T)$ .
\end{thm}

\begin{lem}\label{lem:tightness}
 The sequence $\big(\mathcal{L}(\tilde{u}_{\kappa})\big)_{\kappa >0}$ is tight on $(\mathcal{Z}_T, \mathcal{T}_T)$. 
\end{lem}
\begin{proof} We follow the same line of arguments as invoked in \cite{Majee2020}. Thanks to Lemma \ref{lem:a-priori:2}, ${\rm (iii)}$, the uniform estimates ${\rm (ii)}$ of Theorem \ref{thm:for-tightness} hold true for the sequence $(\tilde{u}_{\kappa})_{\kappa >0}$. Thus, it remains to show that $(\tilde{u}_{\kappa})_{\kappa >0}$ satisfies the Aldous condition 
 in $W^{-1,p^\prime}$. For that, we rewrite
 \eqref{eq:discrete-nonlinear} as 
\begin{align}
 \tilde{u}_{\kappa}(t)&= \hat{u}_0 + \int_0^t \mbox{div}_x \big({\tt A}(x, u_{\kappa}(s),\grad u_{\kappa}(s)) + \vec{F}(u_{\kappa}(s))\big)\,{\rm d}s + \int_0^t U_\kappa(s)\,{\rm d}s + \tilde{M}_\kappa(t) +  \tilde{J}_{\kappa}(t) \notag \\
  &\equiv  \hat{u}_0 + {\tt T}_1^{\kappa}(t) + {\tt T}_2^{\kappa}(t) + {\tt T}_3^{\kappa}(t) + {\tt T}_4^{\kappa}(t)\,. \label{eq:time-cont-p-laplace}
\end{align}
We show that each term in \eqref{eq:time-cont-p-laplace} satisfies \eqref{inq:Aldous-condition} for certain choice of
$\alpha$ and $\zeta$. Clearly, $ \hat{u}_0$ satisfies \eqref{inq:Aldous-condition} for any $\alpha, \zeta$.
For ${\tt T}_2^{\kappa}(t)$, we have, for any  $\theta >0$
\begin{align*}
& \mathbb{E}\Big[  \big\| {\tt T}_2^{\kappa}(\kappa_m + \theta)-{\tt T}_2^{\kappa}(\kappa_m)\big\|_{W^{-1,p^\prime}}\Big] 
 \le C  \mathbb{E}\Big[  \big\| \int_{\kappa_m}^{\kappa_m + \theta} U_\kappa(s)\,{\rm d}s \big\|_{L^2}\Big] \\
 &\le  C  \mathbb{E}\Big[ \int_{\kappa_m}^{\kappa_m + \theta} \|U_\kappa(s) \big\|_{L^2}\,{\rm d}s\Big]\le C \theta^{\frac{1}{2}} \Big( \mathbb{E}\Big[ \int_0^T \|U(s)\|_{L^2}^2\,{\rm d}s\Big] \Big)^{\frac{1}{2}} \le C  \theta^{\frac{1}{2}}\,,
\end{align*}
where $(\kappa_m)$ is a sequence of $\mathbb{F}$- stopping times  with $\kappa_m \le T$.
Thanks to the assumption \ref{A2}, ${\rm ii)}$ and Lemma \ref{lem:a-priori:2} along with H\"{o}lder's inequality, one has 
\begin{align*}
& \mathbb{E}\Big[  \big\| {\tt T}_1^{\kappa}(\kappa_m + \theta)-{\tt T}_1^{\kappa}(\kappa_m)\big\|_{W^{-1,p^\prime}}\Big] \\
 &=\mathbb{E}\Big[ \big\|\int_{\kappa_m}^{\kappa_m + \theta} \mbox{div}_x \big(  {\tt A} (\cdot, u_{\kappa}(s),\grad u_{\kappa}(s))+ \vec{F}(u_{\kappa}(s))\big)\,{\rm d}s\big\|_{W^{-1,p^\prime}}\Big] \notag \\
 & \le \mathbb{E}\Big[\int_{\kappa_m}^{\kappa_m + \theta}\big\|  {\tt A} (\cdot, u_{\kappa}(s),\grad u_{\kappa}(s))+ \vec{F}(u_{\kappa}(s))\big\|_{L^{p^\prime}}\,{\rm d}s\Big] \\
 & \le C\mathbb{E}\Big[\int_{\kappa_m}^{\kappa_m + \theta}\Big( \big\| u_{\kappa}(s)\big\|_{W_0^{1,p}}^{p}+   \|u_\kappa(s)\|_{L^{p^\prime}}^{p^\prime}+ \|K_2\|_{L^{p^\prime}}^{p^\prime} \Big)^\frac{1}{p^\prime}\,{\rm d}s\Big] \\
 & \le C \theta^\frac{1}{p}\Big\{ \mathbb{E}\Big[ \int_0^T\Big(\big\| u_{\kappa}(s)\big\|_{W_0^{1,p}}^{p}+   \|u_\kappa(s)\|_{L^{p^\prime}}^{p^\prime}+ \|K_2\|_{L^{p^\prime}}^{p^\prime}\Big)\,{\rm d}s\Big] \Big\}^\frac{1}{p^\prime}
 \le C \theta^\frac{1}{p}\,.
\end{align*}
Thus, \eqref{inq:Aldous-condition} is satisfied by  ${\tt T}_1^{\kappa}(t)$ for $\alpha=1$ and $\zeta=\frac{1}{p}$.
Furthermore, by using the embedding  $W_0^{1,p}\hookrightarrow L^2\hookrightarrow W^{-1,p^\prime}$, It\^{o}-L\'{e}vy isometry, the assumption \ref{A6} and ${\rm (iii)}$ of Lemma \ref{lem:a-priori:2}, we obtain
\begin{align*}
  \mathbb{E}\Big[\big\| {\tt T}_4^{\kappa}(\kappa_m +\theta)-{\tt T}_4^{\kappa}(\kappa_m)\big\|_{W^{-1,p^\prime}}^2\Big] 
 & \le C \mathbb{E}\Big[\Big\|  \int_{\kappa_m}^{\kappa_m + \theta} \int_{\tt E} \eta(\cdot,\bar{u}_{\kappa}(s);z) \widetilde{N}({\rm d}z,{\rm d}s)\Big\|_{L^2}^2\Big] \\
 & \le C \mathbb{E}\Big[ \int_{\kappa_m}^{\kappa_m + \theta} \int_{\tt E} \big\| \eta(\cdot,\bar{u}_{\kappa}(s);z)\big\|_{L^2}^2 \,m({\rm d}z)\,{\rm d}s\Big] \\
 & \le C c_{\gamma} \mathbb{E}\Big[ \int_{\kappa_m}^{\kappa_m + \theta} \big( \|g\|_{L^2}^2 + \|g\|_{L^\infty}^2\big\| \bar{u}_{\kappa}(s)\big\|_{L^2}^2 \big)\,{\rm d}s\Big]\le C \theta.
\end{align*}
This shows that ${\tt T}_4^{\kappa}(t)$ satisfies \eqref{inq:Aldous-condition} for $\alpha=2$ and $\zeta=1$. Similarly, one can easily show that 
\begin{align*}
  \mathbb{E}\Big[\big\| {\tt T}_3^{\kappa}(\kappa_m +\theta)-{\tt T}_3^{\kappa}(\kappa_m)\big\|_{W^{-1,p^\prime}}^2\Big]
  \le C \theta \Big(|D|+\mathbb{E}\Big[\sup_{s\in [0,T]}\big\| u_{\kappa}(s))\big\|_{L^2}^2\Big] \Big) \le  C \theta.
 \end{align*}
Thus, the sequence $\big(\mathcal{L}(\tilde{u}_{\kappa})\big)_{\kappa >0}$ is tight on $(\mathcal{Z}_T, \mathcal{T}_T)$.
\end{proof}

Let $M_{\bar{\mathbb{N}}}({\tt E} \times [0,T])$ be set of all $\bar{\mathbb{N}}:=\mathbb{N}\cup \{ \infty\}$-valued measures on $({\tt E}\times[0,T],\mathcal{B}({\tt E}\times[0,T]))$ endowed with the $\sigma$-field $\mathcal{M}_{\bar{\mathbb{N}}}({\tt E}\times [0,T])$ generated by the projection maps 
$i_B: M_{\bar{\mathbb{N}}}({\tt E}\times [0,T])\ni \mu\mapsto \mu(B)\in \bar{\mathbb{N}}\quad \forall B \in \mathcal{B}({\tt E}\times [0,T])$. Observe that $M_{\bar{\mathbb{N}}}({\tt E}\times [0,T])$ is a separable metric space. For $\kappa>0$, define $N_{\kappa}({\rm d}z,{\rm d}t):=N({\rm d}z,{\rm d}t)$. Then, by \cite[Theorem 3.2]{Parthasarathy1967}, we arrive at the following conclusion.
\begin{lem}
The laws of the family  $\{ N_{\kappa}({\rm d}z,{\rm d}t)\}$ is tight on $M_{\bar{\mathbb{N}}}({\tt E} \times [0,T])$. Moreover, the set $\{\mathcal{L}(W_\kappa): \kappa >0\}$ is tight on  $C([0,T];\mathcal{U})$, where $W_\kappa=W$ for all $\kappa>0$.
\end{lem}
For control processes $({U}_\kappa)$, we show tightness of the family
 $(\mathcal{L}({U}_\kappa))_{\kappa>0}$ on $\mathbb{X}_{U}:= \big( L^2(0,T;L^2),w\big)$.
\begin{lem}\label{lem:tightness-control}
 The family of the laws of  $({U}_\kappa)$ i.e., $\{ \mathcal{L}(U_{\kappa}): \kappa >0\}$ is tight on $\mathbb{X}_{U}$.
\end{lem}
\begin{proof}
Observe that 
\begin{align*}
\sup_{\kappa>0} \mathbb{E}\Big[ \int_0^T \|U_{\kappa}(t)\|_{L^2}^2\,{\rm d}t\Big] \le \sup_{\kappa>0} \kappa \sum_{k=0}^{N-1} \mathbb{E}\big[ \|U_{k+1}\|_{L^2}^2\big] \le C \mathbb{E}\Big[ \int_0^T \|U(t)\|_{L^2}^2\,{\rm d}t\Big] \le C\,.
\end{align*}
Now, for any $R>0$, the set $B_R:=\{ U\in L^2(0,T; L^2): \|U\|_{L^2(0,T; L^2)}\le R\}$ is relatively compact in $\mathbb{X}_{U}$. Moreover, by Markov's inequality, we see that
\begin{align*}
\mathbb{P}\Big( \|U_\kappa\|_{L^2(0,T;L^2)} >R\Big) \le \frac{ \mathbb{E}\Big[ \int_0^T \|U_{\kappa}(t)\|_{L^2}^2\,{\rm d}t\Big] }{R^2}\le \frac{C}{R^2}\,.
\end{align*}
Hence, the set  $\{ \mathcal{L}(U_{\kappa}): \kappa >0\}$ is tight in $\mathbb{X}_{U}$.
\end{proof}
Thanks to Lemmas \ref{lem:tightness}-\ref{lem:tightness-control}, we infer that the family of the laws of sequence $\big\{(\tilde{u}_{\kappa}, U_{\kappa}, W_\kappa, N_{\kappa}): \kappa >0\big\}$ is tight on $\mathcal{X}_T$, where the space $\mathcal{X}_T$ is defined as 
$$ \mathcal{X}_T:=\mathcal{Z}_T\times \mathbb{X}_U\times C([0,T];\mathcal{U}) \times M_{\bar{\mathbb{N}}}({\tt E} \times [0,T]).$$ We then apply Jakubowski's version of Skorokhod theorem together with \cite[Corollary 2]{Motyl2013}, and \cite[Theorem $D1$]{erika2014} to arrive at the following proposition. 
\begin{prop}\label{prop:skorokhod-representation}
There exist a new probability space $(\bar{\Omega}, \bar{\mathcal{F}}, \bar{\mathbb{P}})$, and  $\mathcal{X}_T$-valued 
random variables \\ $(u_{\kappa}^*, U_{\kappa}^*,W_{\kappa}^*, N_{\kappa}^*)$ and $(u_*, U_*,W_*, N_*)$ defined on  $(\bar{\Omega}, \bar{\mathcal{F}}, \bar{\mathbb{P}})$ satisfying the followings:
\begin{itemize}
 \item [a).] $\mathcal{L}({u}_{\kappa}^*, U_{\kappa}^*, W_{\kappa}^*,N_{\kappa}^*)=\mathcal{L}(\tilde{u}_{\kappa}, U_{\kappa}, W_\kappa, N_{\kappa})$ for all $\kappa >0$,
 \item [b).] $({u}_{\kappa}^*, U_{\kappa}^*, W_{\kappa}^*, N_{\kappa}^*)\goto ({u}_*,U_*, W_*, N_*)$ in $\mathcal{X}_T\quad \bar{\mathbb{P}}$-a.s. as $\kappa \goto 0$,
 \item [c).] $\big( W_{\kappa}^*(\bar{\omega}), N_{\kappa}^*(\bar{\omega})\big)=\big(  W_*(\bar{\omega}), N_*(\bar{\omega})\big)$ for all $\bar{\omega}\in \bar{\Omega}$.
\end{itemize}
\end{prop}
Moreover, there exists a sequence of  perfect functions 
$\phi_{\kappa}:\bar{\Omega}\to\Omega$ such that
\begin{align}
  {u}_{\kappa}^*=\tilde{u}_{\kappa}\circ\phi_{\kappa}\,, \quad U_{\kappa}^*= U_{\kappa} \circ \phi_{\kappa}\,, \quad \mathbb{P}=\bar{\mathbb{P}}\circ \phi_{\kappa}^{-1}\,; \label{eq:perfect-function}
 \end{align}
 see e.g., \cite[Theorem $1.10.4$ \&\ Addendum $1.10.5$]{wellner}. Furthermore, by \cite[Section 9]{erika2014} we infer that $N_{\kappa}^*$ and $N_*$ are time-homogeneous Poisson random measures on ${\tt E}$ with the intensity measure $m({\rm d}z)$, and $W_{\kappa}^*, \kappa >0$ and $W_*$ are $L^2$-valued cylindrical Wiener processes over the stochastic basis
$(\bar{\Omega}, \bar{\mathcal{F}}, \bar{\mathbb{P}},\bar{\mathbb{F}})$, where  the  filtration
 $\bar{\mathbb{F}}:= \big( \bar{\mathcal{F}}_t\big)_{t\in [0,T]}$ is defined by
\begin{align*}
 \bar{\mathcal{F}}_t:= \sigma\big\{({u}_{\kappa}^*(s), U_{\kappa}^*(s),\, W_{\kappa}^*(s),\,N_{\kappa}^*(s),\, {u}_*(s), U_*(s),W_*(s), N_*(s)): 0\le s\le t\big\}, \quad t\in [0,T]. 
\end{align*}
With the help of Proposition \ref{prop:skorokhod-representation} and \eqref{eq:perfect-function}, we define the following stochastic processes on the stochastic basis $(\bar{\Omega}, \bar{\mathcal{F}}, \bar{\mathbb{P}}, \bar{\mathbb{F}})$:
\begin{equation}\label{defi:discrete-variables-new-prob-space}
 \begin{aligned}
  v_k:&=\hat{u}_k\circ \phi_{\kappa}\,, \quad  \bar{U}_k :=U_k \circ \phi_{\kappa}, \quad k=0,1,\cdots, N, \\
   v_{\kappa}(t):&= \sum_{k=0}^{N-1}v_{k+1} {\bf 1}_{[t_k, t_{k+1})}(t),~~t\in[0,T), \quad v_{\kappa}(T)= v_N, \\
  \bar{v}_{\kappa}(t):&= \sum_{k=0}^{N-1} v_k {\bf 1}_{(t_k, t_{k+1}]}(t),~~t\in(0,T], \quad \bar{v}_{\kappa}(0)=u_{0,\kappa}\,, \\
   M_{\kappa}^*(t):&= \int_0^t {\tt h}(\bar{v}_{\kappa}(s)) \,dW_{\kappa}^*(s),~~~t\in [0,T]\,.
 \\
  J_{\kappa}^*(t):&= \int_0^t \int_{\tt E} \eta(x,\bar{v}_{\kappa}(s);z) \widetilde{N}_{\kappa}^*({\rm d}z,{\rm d}s),~~~t\in [0,T]\, \\
   \bar{M}_{\kappa}^*(t)&= \sum_{k=0}^{N-1}\Big( \frac{M_{\kappa}^*(t_{k+1})-M_{\kappa}^*(t_k)}{\kappa}(t-t_k) + M_{\kappa}^*(t_k)\Big) 
    {\bf 1}_{[t_k,t_{k+1})}(t),\quad t\in [0,T), \\
 \bar{J}_{\kappa}^*(t)&= \sum_{k=0}^{N-1}\Big( \frac{J_{\kappa}^*(t_{k+1})-J_{\kappa}^*(t_k)}{\kappa}(t-t_k) + J_{\kappa}^*(t_k)\Big) 
    {\bf 1}_{[t_k,t_{k+1})}(t),\quad t\in [0,T), \\
     \bar{M}_{\kappa}^*(T)&= M_{\kappa}^*(T) \,,\quad  \bar{J}_{\kappa}^*(T)= J_{\kappa}^*(T) \,.
 \end{aligned}
\end{equation}
Thanks to \eqref{eq:perfect-function}, \eqref{defi:discrete-variables-new-prob-space} and \eqref{eq:discrete-nonlinear}, one can easily see that $(v_{k})$ satisfies the following scheme:
  for any $0\le k\le N-1$ and $\bar{\mathbb{P}}$ a.s., 
   \begin{align}
   & v_{k+1}- v_k -\kappa\, \mbox{div}_x \big(  {\tt A}(x, v_{k+1},\nabla v_{k+1}) + \vec{F}(v_{k+1})\big) \notag \\
   & = \kappa \bar{U}_k  +   \int_{t_k}^{t_{k+1}}{\tt h}(v_k)\,dW_\kappa^*(s) +  \int_{\tt E}\int_{t_k}^{t_{k+1}}\eta(x, v_k;z) \widetilde{N}_{\kappa}^*({\rm d}z,{\rm d}t). \label{eq:discrete--new-prob-space} 
   \end{align}
  Moreover, $v_k~(k=0,1,\ldots, N)$ also satisfies the estimate \eqref{a-prioriestimate:1}. Furthermore, one has 
     \begin{align*} 
   & u_{\kappa}^*(t)= \sum_{k=0}^{N-1}\Big( \frac{ v_{k+1}-v_k}{\kappa}(t-t_k) + v_k\Big)
   {\bf 1}_{[t_k,t_{k+1})}(t),~~~ t\in [0,T), \quad 
   u_{\kappa}^*(T)= v_N \,. 
  \end{align*}
  Following the same calculations as invoked in the proof of Lemmas \ref{lem:a-priori:1} and \ref{lem:a-priori:2} together with the first part of Proposition \ref{prop:skorokhod-representation}, we derive the following uniform bounds of $u_\kappa^*$:
 \begin{align}\label{a-priori-3}
 \begin{cases}
    \bar{\mathbb{E}}\Big[  \underset{t\in [0,T]}\sup\,\| {u}_{\kappa}^*(t)\|_{L^2}^2\Big] =  \bar{\mathbb{E}}\Big[  \underset{t\in [0,T]}\sup\,\| {v}_{\kappa}(t)\|_{L^2}^2\Big]  \le \displaystyle C\Big( \|u_0\|_{L^2}^2 +  \bar{ \mathbb{E}}\Big[\int_0^T \|U_\kappa^*(s)\|_{L^2}^2\,{\rm d}s\Big]\Big)\,, \\
  \bar{ \mathbb{E}}\Big[ \big\| \bar{v}_{\kappa}(s)\big\|_{L^2(0,T; L^2)}^2 \Big] \le C\,, \quad  \bar{ \mathbb{E}}\Big[ \big\|{u}_{\kappa}^*\big\|_{L^p(0,T; W_0^{1,p})}^p \Big] \le C\,, \\
   \displaystyle  \bar{ \mathbb{E}}\Big[\int_0^T \|u_\kappa^*(t)-v_\kappa(t)\|_{L^2}^2\,{\rm d}t\Big] \le C \kappa\,, \quad  \bar{ \mathbb{E}}\Big[\int_0^T \|u_\kappa^*(t)-\bar{v}_\kappa(t)\|_{L^2}^2\,{\rm d}t\Big] \le C \kappa\,, \\
     \displaystyle  \bar{ \mathbb{E}}\Big[\int_0^T \|U_\kappa^*(s)\|_{L^2}^2\,{\rm d}s\Big] \le C\,, \quad \bar{ \mathbb{E}}\Big[ \big\|{v}_{\kappa}\big\|_{L^p(0,T; W_0^{1,p})}^p \Big] \le C\,, 
  \end{cases}
\end{align}

\begin{lem}\label{lem:conv-martingale-2}
There exists a constant $C>0$, independent of $\kappa$, such that
 \begin{align*}
 \bar{\mathbb{E}}\Big[\int_0^T \|J_{\kappa}^*(t) -\bar{J}_{\kappa}^*(t)\|_{L^2}^2\,{\rm d}t\Big]\le C \kappa\,, \quad   \bar{\mathbb{E}}\Big[\int_0^T \|M_{\kappa}^*(t) -\bar{M}_{\kappa}^*(t)\|_{L^2}^2\,{\rm d}t\Big]\le C \kappa\,.
 \end{align*}
\end{lem}
\begin{proof}
We will prove the first inequality. The second inequality will follow similarly. 
By using It\^{o}-L\'{e}vy isometry, the assumption \ref{A6},  we have 
 \begin{align*}
   & \bar{\mathbb{E}}\Big[\int_0^T \|J_{\kappa}^*(t) -\bar{J}_{\kappa}^*(t)\|_{L^2}^2\,{\rm d}t\Big] 
 = \sum_{k=0}^{N-1}  \bar{\mathbb{E}}\Big[\int_{t_k}^{t_{k+1}}\|J_{\kappa}^*(t) -\bar{J}_{\kappa}^*(t)\|_{L^2}^2\,{\rm d}t\Big] \\
 &= \sum_{k=0}^{N-1}  \bar{\mathbb{E}}\Big[\int_{t_k}^{t_{k+1}} \Big\| \int_{t_k}^{t} \int_{\tt E} \eta(x,\bar{v}_{\kappa}(s);z)\widetilde{N}_*({\rm d}z,{\rm d}s) - \frac{t-t_k}{\kappa} \int_{t_k}^{t_{k+1}}
  \int_{\tt E} \eta(x,\bar{v}_{\kappa}(s);z)\widetilde{N}_*({\rm d}z,{\rm d}s)\Big\|_{L^2}^2 \,{\rm d}t\Big] \\
 & \le C \sum_{k=0}^{N-1} \int_{t_k}^{t_{k+1}} \bar{\mathbb{E}}\Big[\big\| \int_{t_k}^{t_{k+1}}\int_{\tt E} \eta(x, \bar{v}_{\kappa}(s);z)\widetilde{N}_*({\rm d}z,{\rm d}s)\big\|_{L^2}^2 + \int_{t_k}^{t_{k+1}}\int_{\tt E} \| \eta(x, \bar{v}_{\kappa}(s);z) \big\|_{L^2}^2 m({\rm d}z)\,{\rm d}s
 \Big]\,{\rm d}t \\
 & \le C \sum_{k=0}^{N-1} \int_{t_k}^{t_{k+1}} \bar{\mathbb{E}}\Big[ \int_{t_k}^{t_{k+1}}\int_{\tt E} \Big( \|g\|_{L^2}^2 +  \|g\|_{L^\infty}^2\|\bar{v}_{\kappa}(s)\|_{L^2}^2\Big)\, \gamma^2(z) \,m({\rm d}z)\,{\rm d}s\Big]\,{\rm d}t \\
 & \le C \kappa  \Big( 1 + \bar{\mathbb{E}}\Big[ \int_{0}^{T}\|\bar{v}_{\kappa}(s)\|_{L^2}^2\,{\rm d}s\Big] \Big)  \le C \kappa\,,
 \end{align*}
 where the last inequality follows from \eqref{a-priori-3}\,.
 This completes the proof.
\end{proof}
\subsubsection{\bf Convergence analysis}
  Proposition \ref{prop:skorokhod-representation} gives only  $\bar{\mathbb{P}}$-a.s. convergence of ${u}_{\kappa}^*$. The following Lemma is about the strong convergence of ${u}_{\kappa}^* $---which will play a crucial role in later analysis.
\begin{lem}\label{lem:convergence-1}
There holds the following convergence result.
\begin{itemize}
 \item [(i)] ${u}_{\kappa}^*\goto {u}_*$ in $L^q\big(\bar{\Omega}; L^2(0,T; L^2)\big)$ for all $1\le q<p$. Moreover,  $ v_{\kappa}$ and $ \bar{v}_\kappa $ converge to
 $u_*$ in $L^2\big(\bar{\Omega}; L^2(0,T; L^2)\big)$.
  \item [(ii)] ${u}_{\kappa}^* \stackrel{*}{\rightharpoonup} {u}_*$ in $L_w^2\big(\bar{\Omega}; L^\infty(0,T; L^2)\big)$.
\end{itemize}
\end{lem}
\begin{proof}
We follow a similar lines of arguments reported in \cite[Lemma $3.9$]{Majee2020}. Since the sequence $\{{u}_{\kappa}^*\}$ is uniformly bounded in $L^p(\bar{\Omega}; L^2(0,T; L^2))$, it is equi-integrable in $L^q(\bar{\Omega}; L^2(0,T; L^2))$ for all $1\le q<p$, and therefore, by employing Vitali convergence theorem along with the fact that $\bar{\mathbb{P}}$-a.s., ${u}_{\kappa}^* \goto {u}_*$ in $L^2(0,T;L^2)$
(cf.~ ${u}_{\kappa}^* \goto {u}_*$ in $\mathcal{Z}_T$), we arrive at the first part of ${\rm i)}$.  In particular, since $p>2$, we see that ${u}_{\kappa}^* \goto {u}_*$ in $L^2\big(\bar{\Omega}; L^2(0,T; L^2)\big)$.  Hence the second assertion of ${\rm i)}$ follows from \eqref{a-priori-3}.
The assertion ${\rm ii)}$ follows from \cite[ $iii)$, Lemma $3.9$]{Majee2020}. 
\end{proof}
Regarding the stochastic integral, we have the following lemma. 
\begin{lem}\label{lem:conv-martingale-1}
For all $\phi\in W_0^{1,p}$, there holds
 \begin{align*}
 \lim_{\kappa \goto 0}& \bar{\mathbb{E}}\Big[\int_0^T \Big|\int_0^t \int_{\tt E} \big\langle \eta(\cdot, \bar{v}_{\kappa}(s);z)- \eta(\cdot, u_*(s-);z),
 \phi\big\rangle \widetilde{N}_*({\rm d}z,{\rm d}s)\Big|^2\,{\rm d}t\Big]=0\,,  \notag \\
  \lim_{\kappa \goto 0}& \bar{\mathbb{E}}\Big[\int_0^T \Big| \big\langle \int_0^t \big({\tt h}(\bar{v}_{\kappa}(s))- {\tt h}(u_*(s))\big)\,d{W}_*(s),
 \phi\big\rangle \Big|^2\,{\rm d}t\Big]=0\,.
\end{align*}
\end{lem}
\begin{proof} 
The first assertion follows from \cite[Lemma $3.10$]{Majee2020}. Since $W_{\kappa}^*=W_*$ is a cylindrical Winner noise on $L^2$, by using \ref{A4} and It\^{o}-isometry together with the second part of ${\rm (i)}$, Lemma \ref{lem:convergence-1}, we get, for any $\phi\in L^2$,
\begin{align*}
& \bar{\mathbb{E}}\Big[\int_0^T \Big| \big\langle \int_0^t \big({\tt h}(\bar{v}_{\kappa}(s))- {\tt h}(u_*(s))\big)\,d{W}_*(s),
 \phi\big\rangle \Big|^2\,{\rm d}t\Big] \\
 & \le \|\phi\|_{L^2}^2 \bar{\mathbb{E}}\Big[ \int_0^T \int_0^t \int_D \sum_{k\ge 1} \big|{\tt h}_k(\bar{v}_{\kappa}(s))- {\tt h}_k(u_*(s))\big|^2\,{\rm d}s \,{\rm d}x\,{\rm d}t \Big] \\
 & \le C(L_\sigma, T, \phi)\, \bar{\mathbb{E}}\Big[ \int_0^T\| \bar{v}_{\kappa}(s)- v_*(s)\|_{L^2}^2\,{\rm d}s\Big] \goto 0 \quad \text{as}~~\kappa \goto 0. 
\end{align*}
Since $W_0^{1,p}\subset L^2$, the assertion holds true for all $\phi \in W_0^{1,p}$. 
\end{proof}
Similar to \cite[Lemma $3.12$]{Majee2020}, it can be deduce that ${u}_\kappa^*(0) \rightharpoonup {u}_*(0)$ in $L^2$. Moreover, by using Lemmas \ref{lem:conv-martingale-2}, \ref{lem:convergence-1} and \ref{lem:conv-martingale-1} together with  \eqref{a-priori-3},
one arrives at the following lemma.
\begin{lem} \label{lem:convergence-allterm}
For all $\phi\in W_0^{1,p}$, the following convergence results hold.
 \begin{align*}
 &\lim_{\kappa\goto 0} \bar{\mathbb{E}}\Big[\big|\big({u}_{\kappa}^*(0)-{u}_*(0),\phi\big)_{L^2}\big|\Big]=0\,, \\
 & \lim_{\kappa \goto 0} \bar{\mathbb{E}}\Big[ \int_0^T \Big| \int_0^t \big\langle {\rm div}_x \big( \vec{F}(\bar{v}_{\kappa}(s))- \vec{F}(u_*(s))\big), \phi\big\rangle\, {\rm d}s \Big|\,{\rm d}t\Big]=0\,, \notag \\
  & \lim_{\kappa \goto 0} \bar{\mathbb{E}}\Big[ \int_0^T \Big|\big\langle \bar{M}_{\kappa}^*(t), \phi \big\rangle- \Big\langle \int_0^t  {\tt h}(u_*(s,\cdot)) \,d{W}_*(s), \phi \Big\rangle\Big|\,{\rm d}t\Big]=0\,,\notag \\
 & \lim_{\kappa \goto 0} \bar{\mathbb{E}}\Big[ \int_0^T \Big|\big\langle  \bar{J}_{\kappa}^*(t), \phi \big\rangle- \Big\langle \int_0^t \int_{\tt E} \eta(\cdot,u_*(s-,\cdot);z) \widetilde{N}_*({\rm d}z,{\rm d}s),
  \phi \Big\rangle\Big|\,{\rm d}t\Big]=0\,. \notag
  \end{align*}
\end{lem}

Since $\bar{\mathbb{P}}$-a.s., $U_\kappa^* \rightharpoonup U_*$ in $L^2(0,T; L^2)$, by recalling the estimate of $U_\kappa^*$ as in \eqref{a-priori-3}, we have, after applying dominated convergence theorem
\begin{align}
 \lim_{\kappa \goto 0} \bar{\mathbb{E}}\Big[ \int_0^T \Big| \int_0^t \big\langle  U_\kappa^*(s)-U_*(s), 
 \phi\big\rangle\, {\rm d}s \Big|\,{\rm d}t\Big]=0 , \quad \forall~\phi \in W_{0}^{1,p}\,.\label{eq:limit-control}
\end{align}

\begin{lem}\label{lem:converegnce-operator-term}
There exists $G\in L^{p^\prime}(\bar{\Omega}\times D_T)^d$ such that for any $\phi\in W_0^{1,p}$
 \begin{align}
 \lim_{\kappa \goto 0} \bar{\mathbb{E}}\Big[ \int_0^T \Big| \int_0^t \big\langle {\rm div}_x  \big( {\tt A}(x, v_\kappa(s),\grad v_{\kappa}(s))  - G(s)\big), 
 \phi\big\rangle\, {\rm d}s \Big|\,{\rm d}t\Big]=0\,.\label{eq:limit-3}
\end{align}
\end{lem}
\begin{proof} Note that, thanks to the assumption \ref{A2},${\rm ii)}$, 
\begin{align*}
\|{\tt A}(x, v_\kappa(s),\grad v_{\kappa}(s)) \|_{L^{p^\prime}}^{p^\prime} \le C\big( \|K_2\|_{L^{p^\prime}}^{p^\prime} + \|v_\kappa(s)\|_{W_0^{1,p}}^{p}\Big)\,.
\end{align*}
Using \eqref{a-priori-3}, we see from the above inequality that the sequence 
$\big\{{\tt A}(x, v_\kappa(s),\nabla v_\kappa)\big\}$ is uniformly bounded in $L^{p^\prime}(\bar{\Omega}\times D_T)^d $ and hence 
 ${\tt A}(x, v_\kappa, \nabla v_\kappa) \rightharpoonup G $ in $L^{p^\prime}(\bar{\Omega}\times D_T)^d $ for some $G \in L^{p^\prime}(\bar{\Omega}\times D_T)^d $. 
 Thus for any $\phi\in W_0^{1,p}$, one has 
\begin{align*}
 \lim_{\kappa \goto 0} \bar{\mathbb{E}}\Big[ \int_0^T \Big| \int_0^t \big\langle {\rm div}_x \big( {\tt A}(x, v_\kappa(s),\grad v_{\kappa}(s)
 - G(s)\big), \phi\big\rangle\, {\rm d}s \Big|\,{\rm d}t\Big]=0\,, 
\end{align*}
i.e., \eqref{eq:limit-3} holds true. 
\end{proof}


\subsubsection{\bf Proof of Theorem \ref{thm:existence-weak}}\label{subsec:existence}
We use Lemmas \ref{lem:convergence-1}, \ref{lem:convergence-allterm}-\ref{lem:converegnce-operator-term} to prove Theorem  \ref{thm:existence-weak} in several steps.
\vspace{.1cm}

\noindent{\bf Step ${\rm (i)}$:} Following \cite{Majee2020}, we define the following three functional. For any $\phi\in W_0^{1,p}$,
\begin{equation*}
\begin{aligned}
 \mathcal{R}_{\kappa}(\tilde{{u}}_{\kappa}, U_\kappa, W, \widetilde{N}; \phi): &= \big( \hat{u}_0, \phi\big)_{L^2} + 
 \int_0^t \big\langle \mbox{div}_x {\tt A}(\cdot, u_\kappa(s),\grad u_{\kappa}(s)) , \phi\big\rangle \,{\rm d}s + \int_0^t \big\langle U_\kappa(s), \phi \big\rangle\,{\rm d}s \notag \\
&  + \int_0^t \big\langle \mbox{div}_x \vec{F}(u_\kappa(s)), \phi\big\rangle\,{\rm d}s + \big\langle \tilde{M}_{\kappa}(t), \phi \big\rangle + \big\langle \tilde{J}_{\kappa}(t), \phi \big\rangle\,, \\
\mathcal{R}_{\kappa}^*({u}_{\kappa}^*, U_\kappa^*, W_\kappa^*, \widetilde{N}_\kappa^*; \phi): &= \big( u_\kappa^*(0), \phi\big)_{L^2} + 
 \int_0^t \big\langle \mbox{div}_x {\tt A}(\cdot, v_\kappa(s),\grad v_{\kappa}(s)) , \phi\big\rangle \,{\rm d}s + \int_0^t \big\langle U_\kappa^*(s), \phi \big\rangle\,{\rm d}s \notag \\
&  + \int_0^t \big\langle \mbox{div}_x \vec{F}(v_\kappa(s)), \phi\big\rangle\,{\rm d}s + \big\langle \bar{M}_{\kappa}^*(t), \phi \big\rangle + \big\langle \bar{J}_{\kappa}^*(t), \phi \big\rangle\,, \\ 
 \mathcal{R}_*({u}_*, U_*, W_*, \widetilde{N}_*; \phi): &= \big( {u}_*(0), \phi\big)_{L^2} + \int_0^t \big\langle \mbox{div}_x G(s), \phi\big\rangle \,{\rm d}s
 + \int_0^t \big\langle \mbox{div}_x \vec{F}(u_*(s)), \phi\big\rangle \,{\rm d}s \\
 & \quad  + \int_0^t \big\langle U_*(s), \phi\big\rangle\, {\rm d}s 
 +  \Big\langle \int_0^t  {\tt h}(u_*(s-,\cdot)) \,dW_*(s), \phi \Big\rangle\, \notag \\
& \hspace{2cm} + \Big\langle \int_0^t \int_{\tt E} \eta(\cdot, u_*(s-,\cdot);z) \widetilde{N}_*({\rm d}z,{\rm d}s), \phi \Big\rangle\,.
\end{aligned}
\end{equation*}
Using Lemmas \ref{lem:convergence-allterm} -\ref{lem:converegnce-operator-term}, ${\rm (i)}$, Lemma \ref{lem:convergence-1} together with \eqref{eq:limit-control}, and 
mimicking the arguments of \cite[{\bf Step $i)$}, Proof of Theorem $2.1$]{Majee2020}, we get, $\bar{\mathbb{P}}$-a.s., for a.e. $t\in [0,T]$, and $\phi\in W_0^{1,p}$,
\begin{align}
 \big( {u}_*(t), \phi \big)_{L^2}& = \big( {u}_*(0), \phi \big)_{L^2} +  \int_0^t \big\langle \mbox{div}_x \big( G(s) + \vec{F} (u_*(s))\big), \phi\big\rangle \,{\rm d}s  + \int_0^t  \big\langle  U_*(s), \phi\big\rangle\,{\rm d}s \notag \\
 & \hspace{1cm} + \Big\langle \int_0^t {\tt h}(u_*(s,\cdot)) \,dW_*(s), \phi \Big\rangle   + \Big\langle \int_0^t \int_{\tt E} \eta(\cdot,u_*(s-,\cdot);z) \widetilde{N}_*({\rm d}z,{\rm d}s), \phi \Big\rangle\,. \label{eq:nonlinear-new} 
\end{align}
Moreover, \eqref{eq:nonlinear-new} holds true for all $t\in [0,T]$ and all $\phi \in W_0^{1,p}$ as ${u}_* \in \mathbb{D}([0,T]; L_w^2)$.
\vspace{.2cm}

\noindent{\bf Step ${\rm (ii)}$:} To identify  the function $G \in L^{p^\prime}(\bar{\Omega}\times D_T)^d$, in this step, we follow the ideas from \cite{Majee2020,Vallet2021}. Using \eqref{eq:identity-0} and taking $ v_{k+1}$ as a test function in \eqref{eq:discrete--new-prob-space}, we have
\begin{align}
& \frac{1}{2} \bar{\mathbb{E}}\Big[ \| v_{k+1}\|_{L^2}^2 - \|v_{k}\|_{L^2}^2 + \|v_{k+1}-v_k\|_{L^2}^2\Big]
 + \kappa\, \bar{\mathbb{E}}\Big[ \int_{D} {\tt A}(x, v_{k+1},\grad v_{k+1})\cdot  \nabla v_{k+1}\,{\rm d}x\Big]\notag  \\
 &\quad  - \bar{\mathbb{E}}\Big[ \int_{D} \int_{t_k}^{t_{k+1}}\bar{U}_{k} v_{k+1}\,{\rm d}s\,{\rm d}x\Big]
  - \bar{\mathbb{E}}\Big[ \Big( \int_{t_k}^{t_{k+1}} {\tt h}(v_k)\,{\rm d}{W}_{\kappa}^*(t),
 v_{k+1}-v_k\Big)_{L^2}\Big] \notag \\
 &  \qquad - \bar{\mathbb{E}}\Big[ \Big( \int_{t_k}^{t_{k+1}}\int_{\tt E} \eta(x, v_k;z)\widetilde{N}_{\kappa}^*({\rm d}z,{\rm d}t),
 v_{k+1}-v_k\Big)_{L^2}\Big] =0\,. \label{eq:discrete-1-new}
\end{align}
Observe that by using Cauchy-Schwartz inequality and  It\^{o}-L\'{e}vy isometry, one has
\begin{align*}
 &  \bar{\mathbb{E}}\Big[ \Big( \int_{t_k}^{t_{k+1}} {\tt h}(v_k)\,d{W}_{\kappa}^*(t),
 v_{k+1}-v_k\Big)_{L^2} + \Big( \int_{t_k}^{t_{k+1}}\int_{\tt E} \eta(x, v_k;z)\widetilde{N}_{\kappa}^*({\rm d}z,{\rm d}t),
 v_{k+1}-v_k\Big)_{L^2}\Big] \notag \\
 & \le  \frac{1}{2}\bar{\mathbb{E}}\Big[ \| v_{k+1}-v_k\|_{L^2}^2\Big]  + \bar{\mathbb{E}}\Big[ 
 \int_{t_k}^{t_{k+1}}\| {\tt h}(v_k)\|_{\mathcal{L}_2(L^2)}^2\,{\rm d}s\Big]  
  + \bar{\mathbb{E}}\Big[ \int_{t_k}^{t_{k+1}}\int_{\tt E} \|\eta(\cdot,v_k;z)\|_{L^2}^2\,m({\rm d}z)\,{\rm d}s\Big]\,.
\end{align*} 
Combining this inequality  in \eqref{eq:discrete-1-new}, and then summing over $k=0,1,\ldots, N-1$ in the resulting inequality, 
we get~(using the fact that $v_N=u_{\kappa}^*(T)$)
\begin{align}
 &\frac{1}{2}\bar{\mathbb{E}}\Big[ \|{u}_{\kappa}^*(T)\|_{L^2}^2\Big] + \bar{\mathbb{E}}\Big[ \int_{D_T}
  {\tt A}(x, v_\kappa(t),\grad v_{\kappa}(t))\cdot \grad v_{\kappa}(t)\,{\rm d}x\,{\rm d}t\Big] 
   -\bar{\mathbb{E}}\Big[ \int_{D_T} U_\kappa^*(s) v_\kappa(s)\,{\rm d}s\,{\rm d}x\Big] \notag \\
  &  -\bar{\mathbb{E}}\Big[ \int_0^T  \|{\tt h}(\bar{v}_{\kappa}(t))\|_{\mathcal{L}_2(L^2)}^2\,{\rm d}t\Big]  -\bar{\mathbb{E}}\Big[ \int_0^T \int_{\tt E} \|\eta(\cdot,\bar{v}_{\kappa}(t);z)\|_{L^2}^2\,m({\rm d}z)\,{\rm d}t\Big]
 \le \frac{1}{2}\|\hat{u}_0\|_{L^2}^2\,.\label{esti:2-new}
\end{align}
An application of It\^{o}-L\'{e}vy formula \cite[similar to Theorem $3.4$]{tudor} to $ \frac{1}{2}\| u_*(t))\|_{2}^2$ in \eqref{eq:nonlinear-new} yields
\begin{align}
 &\frac{1}{2}\bar{\mathbb{E}}\Big[ \| u_*(T)\|_{L^2}^2\Big] + \bar{\mathbb{E}}\Big[ \int_{D_T}
 G \cdot \grad  u_*(s-) \,{\rm d}x\,{\rm d}s\Big] -  \bar{\mathbb{E}}\Big[ \int_{D_T} U_*(s) u_*(s)\,{\rm d}x\,{\rm d}s\Big]\notag \\
 & \hspace{1cm} - \bar{\mathbb{E}}\Big[ \int_0^T \|{\tt h}(u_*(s))\|_{\mathcal{L}_2(L^2)}^2\,{\rm d}s\Big] 
 - \bar{\mathbb{E}}\Big[ \int_0^T\int_{\tt E} \|\eta(\cdot, u_*(s-);z)\|_{L^2}^2\, m({\rm d}z)\,{\rm d}s\Big] 
  = \frac{1}{2} \|u_0\|_{L^2}^2\,.\label{esti:3-new}
\end{align}
We now combine \eqref{esti:2-new} and \eqref{esti:3-new} to obtain
\begin{align}
 & \frac{1}{2}\bar{\mathbb{E}}\Big[ \|{u}_{\kappa}^*(T)\|_{L^2}^2- \| u_*(T)\|_{L^2}^2\Big] 
+ \bar{\mathbb{E}}\Big[ \int_{D_T} {\tt A}(x, v_\kappa(t),\grad v_{\kappa}(t))\cdot \grad v_{\kappa}(t)\,{\rm d}x\,{\rm d}t\Big] \notag \\
 &-\bar{\mathbb{E}}\Big[ \int_{D_T}\Big( U_{\kappa}^*(s) v_{\kappa}(s)- U_*(s) u_*(s) \big)\,{\rm d}x\,{\rm d}s\Big] -\bar{\mathbb{E}}\Big[ \int_0^T  \Big(\|{\tt h}(\bar{v}_{\kappa}(s))\|_{\mathcal{L}_2(L^2)}^2- \|{\tt h}(u_*(s))\|_{\mathcal{L}_2(L^2)}^2\Big)\,{\rm d}s\Big] \notag \\
& \quad -\frac{1}{2}\bar{\mathbb{E}}\Big[ \int_0^T \int_{\tt E} \Big(\|\eta(\bar{v}_{\kappa}(s);z)\|_{L^2}^2- \|\eta(v_*(s-);z)\|_{L^2}^2\Big)\, m({\rm d}z)\,{\rm d}s\Big] \notag \\
&\hspace{2cm} \le  \bar{\mathbb{E}}\Big[ \int_{D_T}G \cdot \grad u_* \,{\rm d}x\,{\rm d}t\Big] + \frac{1}{2} \big( \|\hat{u}_0\|_{L^2}^2-\|u_0\|_{L^2}^2\big)\,. \label{esti:new-3-0}
\end{align}
Note that 
\begin{align*}
 \displaystyle \liminf_{\kappa >0} \bar{\mathbb{E}}\Big[ \|{u}_{\kappa}^*(T)\|_{L^2}^2- \| u_*(T)\|_{L^2}^2\Big] \ge 0\,, \quad \|\hat{u}_0-u_0\|_{L^2}\goto 0~~\text{as}~~\kappa \goto 0.
\end{align*}
Since $\bar{\mathbb{P}}$-a.s.,  $U_\kappa^* \rightharpoonup U_*$ in $L^2(0,T;L^2)$, $v_\kappa \rightarrow u_*$ in $L^2(\bar{\Omega}; L^2(0,T;L^2))$ and 
$ \displaystyle \bar{\mathbb{E}}\Big[\int_0^T \|U_\kappa^*(s)\|_{L^2}^2\,{\rm d}s\Big] \le C$ for some constant $C>0$, independent of $\kappa>0$, one can easily conclude that
\begin{align*}
\bar{\mathbb{E}}\Big[ \int_{D_T}\Big( U_{\kappa}^*(s) v_{\kappa}(s)- U_*(s) u_*(s) \big)\,{\rm d}x\,{\rm d}s\Big]\goto 0 ~~\text{as}~~\kappa \goto 0.
\end{align*}
Observe that, in view of the assumption \ref{A4}
\begin{align*}
& \bar{\mathbb{E}}\Big[\int_0^T\| {\tt h}(\bar{v}_{\kappa}(t))-{\tt h}(u_*(t))\|_{\mathcal{L}_2(L^2)}^2\,{\rm d}t\Big]= 
\bar{\mathbb{E}}\Big[\int_0^T \sum_{n=1}^\infty \| \big({\tt h}(\bar{v}_{\kappa}(t))-{\tt h}(u_*(t))\big)e_n\|_{L^2}^2\,{\rm d}t \Big] \\
&= \bar{\mathbb{E}}\Big[\int_0^T \Big( \sum_{n=1}^\infty \big| {\tt h}_n(\bar{v}_{\kappa}(t))-{\tt h}_n(u_*(t))\big|^2\Big)\,{\rm d}x\,{\rm d}t\Big] \le C(L_\sigma) 
 \bar{\mathbb{E}}\Big[\int_0^T \| \bar{v}_{\kappa}(t)- u_*(t)\|_{L^2}^2\,{\rm d}t \Big]\,.
\end{align*}
Hence, $${\tt h}(\bar{v}_{\kappa}) \rightarrow {\tt h}(u_*) \quad \text{in}\quad L^2(\bar{\Omega}\times[0,T];\mathcal{L}_2(L^2))$$ which then implies that 
$$\displaystyle \bar{\mathbb{E}}\Big[\int_0^T \|{\tt h}(\bar{v}_{\kappa}(t))\|_{\mathcal{L}_2(L^2)}^2\,{\rm d}t\Big] \goto 
 \bar{\mathbb{E}}\Big[\int_0^T \|{\tt h}(u_*(t))\|_{\mathcal{L}_2(L^2)}^2\,{\rm d}t\Big].$$
Similarly, thanks to \ref{A5}-\ref{A6} and ${\rm (i)}$ of Lemma \ref{lem:convergence-1}, we get
\begin{align*}
\displaystyle  \bar{\mathbb{E}}\Big[\int_0^T \int_{\tt E}\|\eta(\cdot,\bar{v}_{\kappa}(t);z)\|_{L^2}^2\,m({\rm d}z)\,{\rm d}t\Big] \goto 
 \bar{\mathbb{E}}\Big[\int_0^T \int_{\tt E}\|\eta(\cdot,u_*(t-);z)\|_{L^2}^2\,m({\rm d}z)\,{\rm d}t\Big]\,.
\end{align*}
We combine the above estimates in \eqref{esti:new-3-0} to finally obtain 
\begin{align*}
 \limsup_{\kappa >0} \bar{\mathbb{E}}\Big[\int_{D_T} {\tt A}(x, v_\kappa(t),\grad v_{\kappa}(t))\cdot \grad v_{\kappa}(t)\,{\rm d}x\,{\rm d}t\Big] 
 \le  \bar{\mathbb{E}}\Big[ \int_{D_T} G \cdot \grad  u_* \,{\rm d}x\,{\rm d}t\Big]\,. 
\end{align*}
Following the proof of \cite[Lemma $2.17$, pp. 287-289]{Vallet2021}, we have 
\begin{align*}
 \bar{\mathbb{E}}\Big[ \int_{D_T} G \cdot \grad  u_* \,{\rm d}x\,{\rm d}t\Big] \le  \liminf_{\kappa >0} \bar{\mathbb{E}}\Big[\int_{D_T} {\tt A}(x, v_\kappa(t),\grad v_{\kappa}(t))\cdot \grad v_{\kappa}(t)\,{\rm d}x\,{\rm d}t\Big]\,.
\end{align*}
As a consequence, for any $v\in L^p(\bar{\Omega}\times D_T)$,
\begin{align*}
  \lim_{\kappa >0} \bar{\mathbb{E}}\Big[\int_{D_T} {\tt A}(x, v_\kappa(t),\grad v_{\kappa}(t))\cdot \big(\grad v_{\kappa}(t)-v\big)\,{\rm d}x\,{\rm d}t\Big]= 
 \bar{\mathbb{E}}\Big[ \int_{D_T} G \cdot \big(\grad  u_*-v\big) \,{\rm d}x\,{\rm d}t\Big]\,. 
 \end{align*}
 Moreover, in \cite{Vallet2021}, it has shown that the sequence $\big({\tt A}(x, v_\kappa, \nabla v_\kappa)\big)$ converges weakly~(along a subsequence if necessary) in $L^{p^\prime}$ to ${\tt A}(x, u_*, \nabla u_*)$
for almost every $(\bar{\omega},t)\in \bar{\Omega}\times [0,T]$ and 
\begin{align*}
 \bar{\mathbb{E}}\Big[\int_{D_T} {\tt A}(x, u_*,\grad u_*)\cdot \big(\grad u_*-v\big)\,{\rm d}x\,{\rm d}t\Big] \le \liminf_{\kappa\goto 0}
 \bar{\mathbb{E}}\Big[\int_{D_T} {\tt A}(x, v_\kappa(t),\grad v_{\kappa}(t))\cdot \big(\grad v_{\kappa}(t)-v\big)\,{\rm d}x\,{\rm d}t\Big]\,.
\end{align*}
Hence  for any  $v\in L^p(\bar{\Omega}\times D_T)$, one has
\begin{align*}
 \bar{\mathbb{E}}\Big[\int_{D_T} {\tt A}(x, u_*,\grad u_*)\cdot \big(\grad u_*-v\big)\,{\rm d}x\,{\rm d}t\Big] \le \bar{\mathbb{E}}\Big[ \int_{D_T} G \cdot \big(\grad  u_*-v\big) \,{\rm d}x\,{\rm d}t\Big]\,. 
\end{align*}
Therefore it follows that $G= {\tt A}(\cdot, u_*, \nabla u_*)$.

\vspace{.2cm}

\noindent{\bf Step ${\rm (iii)}$:}  With $ G={\tt A}(x\cdot, u_*,\nabla u_*)$, \eqref{eq:nonlinear-new} then takes the form 
\begin{align*}
 \big( {u}_*(t), \phi \big)_{L^2}& = \big( {u}_*(0), \phi \big)_{L^2} +  \int_0^t \big\langle \mbox{div}_x \big( {\tt A}(x\cdot, u_*(s),\nabla u_*(s)) + \vec{F} (u_*(s))\big), \phi\big\rangle \,{\rm d}s  + \int_0^t  \big\langle  U_*(s), \phi\big\rangle\,{\rm d}s \notag \\
 & \hspace{0.5cm} + \Big\langle \int_0^t{\tt h}(u_*(s,\cdot)) \,dW_*(s), \phi \Big\rangle   + \Big\langle \int_0^t \int_{\tt E} \eta(\cdot,u_*(s-,\cdot);z) \widetilde{N}_*({\rm d}z,{\rm d}s), \phi \Big\rangle \quad \forall\, \phi \in W_0^{1,p} \,.
\end{align*}
Note that, since $\Pi_{\kappa}U \goto U$ in $L^2(0,T;L^2)$, one can easily see that (from the definition of $U_\kappa$), $\mathbb{P}$-a.s.,
 \begin{align*}
 \|U_\kappa -U\|_{L^2(0,T;L^2)}\goto 0\,.
 \end{align*}
 This implies that $\mathcal{L}(U_\kappa)=\mathcal{L}(U)$. 
 Moreover, since $\mathcal{L}(U_{\kappa}^*)= \mathcal{L}(U_{\kappa})$ on $L^2(0,T;L^2)$ and 
$\bar{\mathbb{P}}$-a.s., ~$U_{\kappa}^* \goto U_*$  weakly in $L^2(0,T;L^2)$,
we see that  $\mathcal{L}(U_*)=\mathcal{L}(U)$ on $L^2(0,T;L^2)$  and $\bar{\mathbb{E}}[\|U_*\|_{L^2(0,T;L^2)}]< + \infty$. Furthermore, in view of Proposition \ref{prop:skorokhod-representation} and the 
 estimates in \eqref{a-priori-3}, one can easily arrive at \eqref{esti:bound-weak-solun}. Hence 
the $8$-tuple $\bar{\pi}:=\big( \bar{\Omega}, \bar{\mathcal{F}},  \bar{\mathbb{P}}, \bar{\mathbb{F}},W_*, N_*, v_*, U_*\big)$
is a weak martingale solution of \eqref{eq:nonlinear}.  This finishes the proof of Theorem \ref{thm:existence-weak}.

\begin{rem}\label{rem:existence-weak-Brownian}
Let us comment on the solution of \eqref{eq:nonlinear} with Wiener noise only. 
\begin{itemize}
\item[i)] In case of Wiener noise only,  we can show that, under the assumptions \ref{A1}-\ref{A4}, the underlying problem has a weak solution $u\in C([0,T]; W^{-1,p^\prime})\cap C([0,T]; L_w^2)\cap L^p(0,T; W_0^{1,p})$ $\bar{\mathbb{P}}$-a.s. Basically one needs to replace $\mathcal{Z}_T$ by $\mathcal{C}_T$ and mimic the arguments to arrive at the assertion, where 
\begin{align*}
\mathcal{C}_T:= C([0,T]; W^{-1,p^\prime})\cap L^2_w(0,T; W^{1,2}) \cap L^2(0,T; L^2)\cap C([0,T]; L^2_w)\,.
\end{align*}
Moreover, by \cite[Theorem $4.2.5$]{Liu2015}, one can easily deduce that 
$u\in L^2(\Omega; C([0,T];L^2))$. 
\item[ii)] The additive Wiener noise case was studied in \cite{Vallet2021} with the $p$-type growth of ${\tt A}$, where $\frac{2d}{d+1}<p<2$. We remark here that, in this case also, our methodology can be applied for multiplicative as well as additive noise to establish existence of a martingale solution. In this case, there would be a possible change in the statement of ${\rm i)}$ of Lemma \ref{lem:convergence-1}---which could be restated in the following way:

\begin{itemize}
\item[i')]
 ${u}_{\kappa}^*\goto {u}_*$ in $L^q\big(\bar{\Omega}; L^2(0,T; L^2)\big)$ for all $1\le q<2$. Moreover,  $ v_{\kappa}$ and $ \bar{v}_\kappa $ converge to
 $u_*$ in $L^q\big(\bar{\Omega}; L^2(0,T; L^2)\big)$. 
 \end{itemize}
 Moreover,  the statement of Lemma \ref{lem:conv-martingale-1} changes to: for all $\phi\in W_0^{1,p}$, there holds
 \begin{align*}
  \lim_{\kappa \goto 0}& \bar{\mathbb{E}}\Big[\int_0^T \Big|\int_0^t  \big\langle {\tt h}(\bar{v}_{\kappa}(s))- {\tt h}(u_*(s)),
 \phi\big\rangle \,d{W}_*(s)\Big|^q\,{\rm d}t\Big]=0\,.
\end{align*}
With these results in hand, one can deduce the results of Lemma \ref{lem:convergence-allterm} and therefore able to prove existence of a martingale solution. 
In \cite{Vallet2021}, authors have mentioned that the compactness arguments are more subtle than in the case of a multiplicative noise as the stochastic integral does not inherit the spatial regularity of the solution. But the above discussion suggests that use of Aldous  tightness criterion together with the Jakubowski-Skorokhod theorem \cite{Jakubowski1998} overcome the compactness issue i.e., difficulty level remains same as in the multiplicative noise case. 
\end{itemize}
\end{rem}
\begin{rem}
For L\'{e}vy noise case, our methodology fails to give existence of a martingale solution when $p$, the growth of the operator ${\tt A}$, satisfies the condition $\frac{2d}{d+1}<p<2$. As we mentioned in Remark \ref{rem:existence-weak-Brownian}, ${\rm i')}$ holds true.  If we are able to get the following limit
\begin{align}
 \lim_{\kappa \goto 0}& \bar{\mathbb{E}}\Big[\int_0^T \Big|\int_0^t \int_{\tt E} \big\langle \eta(\cdot, \bar{v}_{\kappa}(s);z)- \eta(\cdot, u_*(s-);z),
 \phi\big\rangle \widetilde{N}_*({\rm d}z,{\rm d}s)\Big|^q\,{\rm d}t\Big]=0\,, \label{lim:poisson-special}
 \end{align}
 then we could arrive at the existence result. But due to the nature of L\'{e}vy measure $m({\rm d}z)$ i.e. the assumption \ref{A6}, we are not able to deduce \eqref{lim:poisson-special}.

\end{rem}
 
\subsection {\bf Proof of Theorem \ref{thm:pathwise-uniqueness}}
  To show path-wise uniqueness of weak solutions, we use standard $L^1$-contraction method. We first consider a special function.
For any $ \Upsilon\in C^\infty(\R)$ satisfying 
\begin{align*}
      \Upsilon(0) = 0\,,\quad \Upsilon(-r)= \Upsilon(r)\,,\quad 
      \Upsilon^\prime(-r) = -\Upsilon^\prime(r)\,,\quad \Upsilon^{\prime\prime} \ge 0\,,\quad 
\Upsilon^\prime(r)=
	\begin{cases}
	-1\quad \text{if} ~ r\le -1\,,\\
         \in [-1,1] \quad\text{if}~ |r|<1\,,\\
         +1 \quad \text{if} ~ r\ge 1\,,
         \end{cases}
\end{align*} 
and for any $ \vartheta > 0$, we define 
$\Upsilon_\vartheta:\R \rightarrow \R$ by 
$\Upsilon_\vartheta(r) =  \displaystyle \vartheta \Upsilon(\frac{r}{\vartheta}).$
The approximate functions $\Upsilon_{\vartheta}\in C^\infty(\R)$  satisfies the following bounds:
\begin{align}\label{eq:approx to abosx}
 \displaystyle   |r|-M_1\vartheta \le \Upsilon_\vartheta(r) \le |r|,~~~
  |\Upsilon_\vartheta^{\prime\prime}(r)| 
 \le \frac{M_2}{\vartheta} {\bf 1}_{|r|\le \vartheta},~~\text{with}~~
 M_1:=\underset{|r|\le 1}\sup\,\big | |r|-\Upsilon(r)\big |,~~M_2:= \underset{|r|\le 1}\sup\,|\Upsilon^{\prime\prime} (r)|\,.
\end{align} 
Let $\big(\Omega, \mathcal{F}, \mathbb{P}, \mathbb{F},  
W, N, u_1, U\big)$ and $\big(\Omega, \mathcal{F}, \mathbb{P}, \mathbb{F},  
W, N, u_2, U\big)$ be two weak solution of \eqref{eq:nonlinear}.
We apply It\^{o}-L\'{e}vy formula to $ \displaystyle \int_{D}\Upsilon_\vartheta \big( u_1(t)-u_2(t)\big)\,{\rm d}x$, and use the integration by parts formula to have
\begin{align}
 &\mathbb{E}\Big[ \int_{D}\Upsilon_\vartheta \big(u_1(t)-u_2(t)\big)\,{\rm d}x \Big]\notag \\
 & = - \mathbb{E}\Big[ \int_{0}^t \int_{D} \big(  {\tt A}(x, u_1,\grad u_1) - {\tt A}(x, u_2,\grad u_2)\big)
 \cdot  \big(\grad u_1(s)- \grad u_2(s)\big)\Upsilon_{\vartheta}^{\prime\prime} \big(u_1(s)-u_2(s)\big) \,{\rm d}x\,{\rm d}s \Big]\notag \\
 &\quad  -\mathbb{E}\Big[ \int_{0}^t \int_{D} \big(  \vec{F}(u_1) - \vec{F}(u_2)\big)
 \cdot  \big(\grad u_1(s)- \grad u_2(s)\big)\Upsilon_{\vartheta}^{\prime\prime} \big(u_1(s)-u_2(s)\big) \,{\rm d}x\,{\rm d}s\Big] \notag \\
 & +\mathbb{E}\Big[ \int_{0}^{t}\int_{\tt E}\int_{D}\int_0^1  \Upsilon_\vartheta^{\prime\prime}\Big( u_1(s-)-u_2(s-) + \lambda 
 \big( \eta(x,u_1(s-);z)-\eta(x,u_2(s-);z)\big)\Big) \notag \\
 & \hspace{3cm} \times (1-\lambda)  \big( \eta(x,u_1(s-);z)-\eta(x,u_2(s-);z)\big)^2\,{\rm d}\lambda\,{\rm d}x\,m({\rm d}z)\,{\rm d}s\Big] \notag \\
 & +  \frac{1}{2} \mathbb{E}\Big[ \int_{0}^{t}\int_{D}\Upsilon_\vartheta^{\prime\prime}\big( u_1(s) -u_2(s)\big) \sum_{k=1}^\infty\big| {\tt h}_k(u_1(s)-{\tt h}_k(u_2(s)\big|^2\,{\rm d}s\,{\rm d}x\Big] 
 \equiv \sum_{i=1}^4\mathcal{B}_i \,.\label{eq:uniqueness-1}
\end{align}
In view of ${\rm (i)}$ of \ref{A2} and the fact that $\Upsilon_{\vartheta}^{\prime \prime}\ge 0$, one has
\begin{align*}
\mathcal{B}_1\le  \mathbb{E}\Big[   \int_{0}^t \int_{D}  \big(  {\tt A}(x, u_2,\grad u_2) - {\tt A}(x, u_1,\grad u_2)\big)\cdot \nabla(u_1(s)-u_2(s))
 \Upsilon_{\vartheta}^{\prime\prime} \big(u_1(s)u_2(s)\big)\,{\rm d}x\,{\rm d}s\Big] \,.
\end{align*}
 By using  ${\rm (iii)}$ of \ref{A2} and \eqref{eq:approx to abosx}, we have
 $\mathbb{P}$-a.s., 
\begin{align*}
&  \big(  {\tt A}(x, u_2,\grad u_2) - {\tt A}(x, u_1,\grad u_2)\big)\cdot \nabla(u_1(s)-u_2(s))
 \Upsilon_{\vartheta}^{\prime\prime} \big(u_1(s)u_2(s)\big) \\
& \le \big( C_4 |\nabla u_2|^{p-1} + K_3(x)\big) |u_1 -u_2|\,\big|\nabla( u_1-u_2)\big| \frac{M_2}{\vartheta} \textbf{1}_{\{|u_1-u_2| \le \vartheta\}}   \goto 0 \quad (\vartheta \goto 0)
\end{align*}
for a.e. $(t,x)\in D_T$. Moreover, Young's inequality gives 
\begin{align*}
 & \Big|  \big(  {\tt A}(x, u_2,\grad u_2) - {\tt A}(x, u_1,\grad u_2)\big)\cdot \nabla(u_1(s)-u_2(s))
 \Upsilon_{\vartheta}^{\prime\prime} \big(u_1(s)u_2(s)\big) \Big| \\
 & \le  C \Big( |\nabla(u_1-u_2)|^p + |\nabla u_2|^p + |K_3(x)|^{p^\prime}\Big) \in L^1(\Omega \times D_T)\,.
\end{align*}
 We use dominated convergence theorem to infer that $\mathcal{B}_1\goto 0$ as $\vartheta \goto 0$. Following the arguments as done in \cite[pp 18-19]{Majee2020}, we have $\mathcal{B}_2,~\mathcal{B}_3\goto 0$ as $\vartheta \goto 0$. Furthermore, thanks to the assumption \ref{A4} and \eqref{eq:approx to abosx}, we see that
 \begin{align*}
 \mathcal{B}_4\le C(L_{\sigma})  \mathbb{E}\Big[\int_0^t \int_D |u_1-u_2|^2 \Upsilon_{\vartheta}^{\prime\prime} \big(u_1(s)-u_2(s)\big)\,{\rm d}x\,{\rm d}s\Big] \le C\vartheta \,t \goto 0 ~~~\text{as}~~\vartheta \goto 0\,.
 \end{align*}
In view of above convergence results, we pass to the limit in \eqref{eq:uniqueness-1} to conclude
\begin{align}
 \mathbb{E}\Big[\int_{D}\big| u_1(t,x)-u_2(t,x)\big|\,{\rm d}x\Big]=0 \notag
\end{align}
---which then yields path-wise uniqueness of weak solutions.
\subsection{Proof of Theorem \ref{thm:existence-strong}}
In view of Theorem \ref{thm:existence-weak} there exists a martingale solution 
$$u\in L^2(\Omega; \mathbb{D}([0,T]; L^2_w))\cap L^p(\Omega; L^p(0,T; W_0^{1,p}))$$ and by Theorem \ref{thm:pathwise-uniqueness} it is path-wise unique.  Hence the existence of  a unique strong solution of \eqref{eq:nonlinear} follows from \cite[Theorem 2]{Ondrejat2004}.

\section{Existence of optimal control: proof of Theorem \ref{thm:existence-optimal-control}}
In this section, we prove existence of a weak optimal solution of \eqref{eq:control-problem} i.e., Theorem \ref{thm:existence-optimal-control} in four steps using standard variational approach.
\vspace{.1cm}

\noindent{\bf Step ${\rm (a)}$:} By Theorem \ref{thm:existence-weak}, there exists a
weak solution of \eqref{eq:nonlinear} with $U=0$ satisfying the estimate \eqref{esti:bound-weak-solun}. Since $ {\tt u}_{\rm det}\in L^p(0,T; W_0^{1,p})$ is a given deterministic profile, and  ${\tt \Psi}$ is a
Lipschitz continuous function, one can easily deduce that 
$\Lambda:=  \inf_{\pi \in \mathcal{U}_{\rm ad}^w(u_0; T)} \mathcal{\pmb J}(\pi)< + \infty.$ Let $\pi_n=\big( \Omega_n,  \mathcal{F}_n,  \mathbb{P}_n,\mathbb{F}_n=\{ \mathcal{F}_t^n\}, W_n, N_n, u_n, U_n\big)\in \mathcal{U}_{\rm  ad}^w(u_0;T)$ be a 
minimizing sequence of weak admissible solutions ~(which exists since $\Lambda<\infty$) such that
$\Lambda= \displaystyle \lim_{n\goto \infty}\mathcal{\pmb J}(\pi_n)$. Then, $\mathbb{P}_n$-a.s. and for all $t\in [0,T]$,
\begin{align*}
 u_n(t) &= u_0 + \int_0^t {\rm div}_x \big({\tt A}(x, u_n(s), \nabla u_n(s))+ \vec{F}(u_n(s))\big)\,{\rm d}s + \int_0^t U_n(s)\,{\rm d}s  \notag \\
 & \quad + \int_0^t {\tt h}(u_n(s))\,dW_n(s)
 + \int_0^t \int_{\tt E} \eta(x,u_n;z)\widetilde{N}_n({\rm d}z,{\rm d}s)
 := u_0 + \sum_{i=1}^4{\tt T}_{i,n}(t)\,.
\end{align*}
Furthermore, since $\Lambda $ is finite, there exists a constant $C>0$, independent of $n$ such that the following estimates (uniform in $n$) hold:
\begin{align}\label{esti:apriori-solun-n}
\displaystyle  \sup_{n}\mathbb{E}_n\Big[ \int_0^T \|U_n(t)\|_{L^2}^2\,{\rm d}t\Big]\le C, \quad 
 \displaystyle \sup_{n}\mathbb{E}_n \Big[\sup_{0\le t\le T} \| u_n(t)\|_{L^2}^2 + \int_0^T \| u_n(t)\|_{W_0^{1,p}}^p\,{\rm d}t \Big] \le C\,.
\end{align}
\noindent{\bf Step ${\rm (b)}$:} We use the same arguments as  done in the proof of Lemma \ref{lem:tightness} and apply Theorem \ref{thm:for-tightness} along with the uniform-estimates \eqref{esti:apriori-solun-n}, to establish the tightness of $\{ \mathcal{L}(u_n)\}$ on $(\mathcal{Z}_T,\mathcal{T}_T)$. 
Moreover, repeating the arguments of Lemma \ref{lem:tightness-control} and using the  tightness of the families  $\{ \mathcal{L}(N_n({\rm d}z,{\rm d}t))\}$ and $\{\mathcal{L}(W_n)\}$ on $M_{\bar{\mathbb{N}}}({\tt E} \times [0,T])$ and $C([0,T];\mathcal{U})$ respectively, we see that the family $\{ \mathcal{L}(u_n,U_n, W_n, N_n)\}$ is tight in $\mathcal{X}_T$. Thus, one can apply Jakubowski version of Skorokhod theorem for a non-metric space e.g.,
 \cite[Corollary $2$]{Motyl2013} to ensure existence of a subsequence of $\{n\}$, still denoted it by same $\{n\}$, a probability space $(\Omega^*, \mathcal{F}^*, \mathbb{P}^*)$, and on it $\mathcal{X}_T$-valued 
random variables $({u}_n^*, U_n^*,W_n^*, N_n^*)$ and $({u}^*, U^*, W^*, N^*)$ such that 
\begin{itemize}
 \item [(1)] $\mathcal{L}({u}_n^*, U_n^*, W_n^*, N_n^*)=\mathcal{L}(u_n, U_n, W_n, N_n)$ for all $n\in \mathbb{N}$,
 \item [(2)] $({u}_n^*, U_n^*, W_n^*, N_n^*)\goto ({u}^*,U^*, W^*, N^*)$ in $\mathcal{X}_T\quad \mathbb{P}^*$-a.s. ~~~~~$(n\goto \infty)$, 
 \item [(3)] $ \big( W_n^*(\omega^*),N_n^*(\omega^*)\big)=  \big( W^*(\omega^*),N^*({\omega}^*)\big)$ for all ${\omega}^*\in {\Omega}^*$.
\end{itemize}
Moreover, there exists a sequence of
 perfect functions $\phi_{n}:\Omega^*\to\Omega_n$ such that
 \begin{align*}
  {u}_{n}^*=u_n\circ\phi_{n}\,, \quad U_{n}^*= U_{n} \circ \phi_{n}\,, \quad \mathbb{P}_n={\mathbb{P}}^*\circ \phi_{n}^{-1}\,, 
 \end{align*}
 and the uniform estimate \eqref{esti:apriori-solun-n} holds true for $u_n^*$ and $U_n^*$. Furthermore, $u_n^*$ satisfies the following PDE: $\mathbb{P}^*$-a.s., 
\begin{align}
 u_n^*(t) &= u_0 + \int_0^t {\rm div}_x\big( {\tt A}(x, u_n^*(s), \nabla u_n^*(s))+ \vec{F}(u_n^*(s))\big)\,{\rm d}s + \int_0^t U_n^*(s)\,{\rm d}s  \notag \\
 & \qquad + \int_0^t {\tt h}(u_n^*(s))\,dW_n^*(s)
 + \int_0^t \int_{\tt E} \eta(x, u_n^*;z)\widetilde{N}_n({\rm d}z,{\rm d}s)\,, \quad \forall~t\in [0,T] \,. \label{eq:p-laplace-change-variable-n}
\end{align}
\noindent{\bf Step ${\rm (c)}$:} Let $\mathbb{F}^*$ be the natural filtration of $({u}_n^*,U_n^*,W_n^*, N_n^*, {u}^*, U^*, W^*, N^*)$.
Then, on stochastic basis$(\Omega^*, \mathcal{F}^*, \mathbb{P}^*, \mathbb{F}^*)$,  $W_n^*$ and $W^*$ are $L^2$-valued cylindrical Wiener processes, and  $N_n^*$ and $N^*$ are time-homogeneous Poisson random measures on ${\tt E}$ with intensity measure $m({\rm d}z)$; see \cite[Section 9]{erika2014}. Similar to proof of Lemmas \ref{lem:convergence-1}, 
\ref{lem:conv-martingale-1} and \ref{lem:converegnce-operator-term} and the arguments of Subsection \ref{subsec:existence}, one can pass to the limit in \eqref{eq:p-laplace-change-variable-n}
and conclude that the limit process $u^*$ is a  $L^2$-valued $\mathbb{F}^*$-predictable stochastic process with
 $u^* \in L^2\big(\Omega^*; \mathbb{D}(0,T; L_w^2)\big)\cap L^p\big( \Omega^*; L^p(0,T;W_0^{1,p})\big)$ satisfying the following:
 $\mathbb{P}^*$-a.s. and for all $t\in [0,T]$
\begin{align*}
 \big( {u}^*(t), \phi \big)_{L^2}& = \big( {u}_0, \phi \big)_{L^2} +  \int_0^t \big\langle \mbox{div}_x \big( {\tt A}(x,\cdot, u^*(s),\nabla u^*(s)) + \vec{F} (u^*(s))\big), \phi\big\rangle \,{\rm d}s  + \int_0^t  \big\langle  U^*(s), \phi\big\rangle\,{\rm d}s \notag \\
 & \hspace{.2cm} + \Big\langle \int_0^t{\tt h}(u^*(s,\cdot)) \,dW^*(s), \phi \Big\rangle   + \Big\langle \int_0^t \int_{\tt E} \eta(\cdot,u^*(s-,\cdot);z) \widetilde{N}^*({\rm d}z,{\rm d}s), \phi \Big\rangle \quad \forall\, \phi \in W_0^{1,p} \,.
\end{align*}
 Hence $\pi^*:=(\Omega^*, \mathcal{F}^*,
\mathbb{P}^*, \mathbb{F}^*, W^*, N^*, u^*, U^*) \in \mathcal{U}_{\rm ad}^w(u_0;T)$. Moreover,  $(u^*, U^*)$ satisfies the estimate \eqref{esti:bound-weak-solun}.

\vspace{.2cm}

\noindent{\bf Step ${\rm (d)}$:} Since $\pi^* \in \mathcal{U}_{\rm ad}^w(u_0;T)$, it is obvious that $\Lambda \le \mathcal{\pmb J}(\pi^*)$. Invoking the convexity of the cost functional (in the control variable) and mimicking the proof of {\bf Step ${\rm IV)}$} of \cite[Theorem $2.2$]{Majee2020}, one arrive at the fact that  $\mathcal{\pmb J}(\pi^*)\le \Lambda$. In other words, $\pi^*$ is a weak optimal solution of \eqref{eq:control-problem},
and $U^*$ is an optimal control. This finishes the proof.
\begin{rem}\label{rem:cont-dependence-initial-cond} 
We make the following remarks regarding the continuously dependency of martingale solution of  \eqref{eq:nonlinear-no-control} on initial data. 
\begin{itemize}
\item[a)] Martingale solution of \eqref{eq:nonlinear-no-control} continuously depends on initial data. Indeed, let $\{u_0^n\}$ be a sequence of $L^2$-valued deterministic functions such that $u_0^n \rightharpoonup u_0$ in $L^2$. Let $(\Omega_n, \mathcal{F}_n, \mathbb{P}_n, \mathbb{F}_n, W_n, N_n, u_n)$ be a martingale solution of \eqref{eq:nonlinear-no-control} with initial datum $u_0^n$. Since $u_0^n \rightharpoonup u_0\in L^2$, one has $\sup_{n\ge 1}\|u_0^n\|_{L^2}^2\le C$. Hence $u_n$ satisfies the uniform estimate \eqref{esti:apriori-solun-n}. Thus, following  {\bf Step ${\rm (A)}$}- {\bf Step ${\rm (D)}$} of the proof of Theorem \ref{thm:existence-optimal-control}, one arrive at the following: there exist 
 \begin{itemize}
 \item[i)] a subsequence $\{u_{n_j}\}_{j=1}^\infty$,
 \item[ii)] a stochastic basis $\big(\bar{\Omega}, \bar{\mathcal{F}},
\bar{\mathbb{P}}, \bar{\mathbb{F}}\big)$,
\item[iii)] time-homogeneous Poisson random measure $\bar{N}_j$ and $\bar{N}$ with $\bar{N}_j(\bar{\omega})=\bar{N}(\bar{\omega});~j\ge 1$ for $\bar{\omega}\in \bar{\Omega}$,
\item[iv)] $L^2$-valued cylindrical Wiener processes $\bar{W}_j$ and $\bar{W}$ with $\bar{W}_j(\bar{\omega})=\bar{W}(\bar{\omega});~j\ge 1$ for $\bar{\omega}\in \bar{\Omega}$,
\item[v)] $\mathcal{Z}_T$-valued Borel measurable random variables $\bar{u}$ and $\{\bar{u}_j\}_{j=1}^\infty$ 
 \end{itemize}
 such that $
 \mathcal{L}(u_{n_j})=\mathcal{L}(\bar{u}_j)$ on $\mathcal{Z}_T$ and $\bar{\mathbb{P}}$-a.s., $\bar{u}_j\goto \bar{u}$ in $\mathcal{Z}_T$. Moreover, the tuple $\big(\bar{\Omega},
\bar{ \mathcal{F}}, \bar{\mathbb{P}}, \bar{\mathbb{F}}, \bar{W},\bar{ N}, \bar{u}\big)$ is a martingale solution of \eqref{eq:nonlinear-no-control} with initial data $u_0$, on the interval $[0,T]$. 
\item[b)] In the case of Brownian noise only, part $a)$ holds true replacing $\mathcal{Z}_T$ by $\mathcal{C}_T$ defined in Remark 
\ref{rem:existence-weak-Brownian}.
\end{itemize}
\end{rem}

\section{Existence of Invariant measure} \label{Existence of Invariant measure}
In this section, we wish to study the existence of an invariant measure for the solution of \eqref{eq:nonlinear-no-control}. In view of Theorem \ref{thm:existence-strong}, there exists path-wise unique strong solution $u\in L^2(\Omega; \mathbb{D}([0,T]; L^2_w))\cap L^p(\Omega; L^p(0,T; W_0^{1,p}))$. We start with the observation that instead of initial condition, one may consider nonlinear SPDE \eqref{eq:nonlinear-no-control} with initial distribution $\rho_0$, a Borel probability measure on $L^2$, and prove existence of a martingale solution. In fact, one can establish the following lemma.
\begin{lem}
Let the assumptions \ref{A1}-\ref{A6} hold, and the initial distribution $\rho_0$ satisfies the condition 
   \begin{align*}
       \int_{L^2}|x|^2 \rho_0({\rm d}x) < R \quad \text{for some} \ R > 0\,. 
   \end{align*}
Then, there exists a weak solution
$\bar{\pi}=\big( \bar{\Omega}, \bar{\mathcal{F}}, \bar{\mathbb{P}},  \{\bar{\mathcal{F}}_t\}, \bar{W},\bar{N}, \bar{u}\big)$ of \eqref{eq:nonlinear-no-control} with $\mathcal{L}(\bar{u}_0) = \rho_0$ in $L^2$, satisfying
the following estimate
\begin{equation*}
 \bar{\mathbb{E}}\Big[\sup_{0\le t\le T} \|\bar{u}(t)\|_{L^2}^2 + \int_0^T \| \bar{u}(t)\|_{W_0^{1,p}}^p\,{\rm d}t \Big] \le C(R,p)
\end{equation*}
for some constant $C>0$.
\end{lem}
We need to recall the definition of uniqueness in law of solutions of \eqref{eq:nonlinear-no-control}
\begin{defi}
Let $\big(\Omega^i, \mathcal{F}^i, \mathbb{P}^i, \mathbb{F}^i, W^i, N^i, u^i\big),~~i=1,2$ be martingale solution of  \eqref{eq:nonlinear-no-control} with $u^i(0)=u_0$. We say that the solution is unique in law if the following holds:
\begin{align*}
{\rm Law}_{\mathbb{P}^1}(u^1)= {\rm Law}_{\mathbb{P}^2}(u^2)\quad \text{on}~~~\mathbb{D}([0,T];L_w^2])\cap L^p([0,T];W_0^{1,p})
\end{align*}
where ${\rm Law}_{\mathbb{P}^i}(u^i)$ for $i=1,2$ are probability measures on $\mathbb{D}([0,T];L_w^2])\cap L^p([0,T];W_0^{1,p}).$
\end{defi}
Since the problem \eqref{eq:nonlinear-no-control} has a path-wise unique martingale solution, a direct application of \cite[Theorems $2~\& ~11$]{Ondrejat2004} yields that the martingale solution is unique in law. Hence, we have the following lemma.
\begin{lem}\label{lem:uniqueness-law}
Under the assumptions \ref{A1}-\ref{A6}, martingale solutions of \eqref{eq:nonlinear-no-control}  is unique in law.
\end{lem}
Recall the family $\{P_t\}$ defined in \eqref{defi:semi-group} i.e., 
\begin{align*}
(P_t \phi)(v):=\mathbb{E}\big[\phi(u(t,v))\big], \quad \phi \in\mathcal{B}(L^2)\,, 
\end{align*}
where $u(t,v)$ is the path-wise unique solution of  \eqref{eq:nonlinear-no-control} with fixed initial data $v\in L^2$. Note that $P_t\phi$ is bounded. Moreover, thanks to Lemma \ref{lem:uniqueness-law}, one can use similar argument as done in \cite[Section $4$]{ABW-2010} to get that 
$P_t \phi \in \mathcal{B}(L^2)$ and $\{P_t\}_{t\ge 0}$ is a semi-group on $\mathcal{B}(L^2)$; see  \cite[Section $9.6$]{peszat}.
\begin{lem}
    The family $\{u(t,v):~ t \ge 0,~ v \in L^2\}$ is Markov. In particular $P_{t+s} = P_t  P_s$ for $t,s \ge 0$.
\end{lem}
\begin{proof}
    Given  $v \in L^2, 0 \le s \le t < \infty,$ we denote by $u(t,s,v)$ the value at time $t$ of the solution to \eqref{eq:nonlinear-no-control} starting at time $s$ from $v$. It is sufficient to prove that for all $t \ge 0$ and $v \in L^2$, the processes $\big( u(h,0,v): h \ge 0 \big)$ and $ \big( u(t+h,t,v): h \ge 0 \big)$  have the same distribution, ~(cf. the arguments on pp. 167-170 of \cite{peszat}). To view this, we note that
    \begin{align*}
        u(t+h,t,v) & = v + \int_{t}^{t+h} \mbox{div}_x \big({\tt A}(x, u(s,t,v),\nabla u(s,t,v)) + \vec{F}(u(s,t,v))\big)\,{\rm d}s \notag \\
  & \hspace{1cm}+\int_{t}^{t+h} {\tt h}(u(s,t,v))\,dW(s) +  \int_{t}^{t+h} \int_{\tt E}\eta(x,u(s,t,v);z)\widetilde{N}({\rm d}z,{\rm d}s) \notag \\
  & = v + \int_{0}^{h} \mbox{div}_x \big({\tt A}(x, u(t+s,t,v),\nabla u(t+s,t,v)) + \vec{F}(u(t+s,t,v))\big)\,{\rm d}s \notag \\
  & \hspace{1cm}+\int_{0}^{h} {\tt h}(u(t+s,t,v))\,dW^{t}(s) +  \int_{0}^{h} \int_{\tt E}\eta(x,u(t+s,t,v);z)\widetilde{N}^{t}({\rm d}z,{\rm d}s)\,,
    \end{align*}
    where $W^{t}(s):= W(t+s) - W(t)$ and $\widetilde{N}^{t}(s):= \widetilde{N}(t+s) -\widetilde{N}(t)$. Clearly, by the stationary property of both the Wiener process $W$ and the Poisson process $\widetilde{N}$, the pair $W^{t}$  and $W$, $\widetilde{N}^{t}$ and $\widetilde{N}$  have the same law. Since, the law of the solution $u(t,s,v)$ of \eqref{eq:nonlinear-no-control} with initial date $v \in L^2$ does not depend on the choice of the probability space, the Wiener process $W$ and the Poisson process $\widetilde{N}$ i.e., for two unique solutions  $u$ and $\bar{u}$ of \eqref{eq:nonlinear-no-control} with two different setup $\big(\Omega, \mathcal{F}, \mathbb{P}, W, N \big)$ and $\big( \bar{\Omega}, \bar{\mathcal{F}}, \bar{\mathbb{P}}, \bar{W}, \bar{N} \big)$ respectively, the law of  $u(t)$ and $\bar{u}(t)$ are same for all $t\ge 0$, we conclude that $u(t+h,t,v)$ and $u(h,0,v)$ have the same law.
\end{proof}
Following the argument as used in \cite[Lemma 5.3]{Majee2023}, we arrive at the following lemma.
\begin{lem}\label{lem:feller-semi-group}
Markov semi-group $\{P_t\}$ is a Feller semi-group i.e., $P_t$ maps $C_b(L^2)$ into itself. 
\end{lem}
\begin{defi} \label{defi:se-weak-bounded}
A function $\phi \in SC(L_w^2)$ if and only if $\phi$ is sequentially continuous with respect to weak topology on $L^2$. By  $SC_b(L_w^2)$, we denote set of all bounded functions in $ SC(L_w^2)$. 
\end{defi}
The following inclusion holds; see \cite{Ferrario2019,Seidler1999}:
\begin{align*}
C_b(L_w^2)\subset SC_b(L_w^2)\subset C_b(L^2)\subset \mathcal{B}_b(L^2)\,.
\end{align*}
\begin{defi}
The family $\{P_t\}$ is said to be sequentially weakly Feller if and only if  $P_t$ maps  $SC_b(L_w^2)$ into itself. 
\end{defi}
\subsection{Proof of Theorem \ref{thm:existence-invariant-measure}}
To prove existence of an invariant measure, we will use the result of Maslowski-Seidler\cite{Seidler1999}, a modification of the Krylov-Bogoliubov technique \cite{Krylov1937}.   We need to show the following; see \cite{Seidler1999}.
\begin{itemize}
\item[i)] Sequentially weakly Feller:  $\{P_t\}$ is sequentially weakly Feller,
\item[ii)] Boundedness in probability: there exists a Borel probability measure $\nu$ on $L^2$ and $T_{0} \ge 0$ such that for any $\eps>0$, there exists $R\equiv R(\eps)>0$ such that 
\begin{align*}
\sup_{T\ge T_{0}}\frac{1}{T} \int_0^T P_{t}^{*} \nu \Big\{ \|x\|_{L^2}>R\Big\}\,{\rm d}t < \eps.
\end{align*}
\end{itemize}
\noindent{Proof of ${\rm i)}$}. Let $0<t\le T$, $v\in L^2$ and $\phi\in SC_b(L_w^2)$ be fixed. Let $\{v_n\}$ be a $L^2$-valued sequence such that $v_n\rightharpoonup v$ in $L^2$. Let $u_n$ be a strong solution of  \eqref{eq:nonlinear-no-control}, defined on the stochastic basis $(\Omega, \mathcal{F}, \mathbb{P}, \mathbb{F})$, on $[0,T]$ with initial datum $v_n$.  Then by Remark \ref{rem:cont-dependence-initial-cond}, there exist a subsequence $\{u_{n_j}\}_{j=1}^\infty$,  a stochastic basis $\big(\bar{\Omega}, \bar{\mathcal{F}},
\bar{\mathbb{P}}, \bar{\mathbb{F}}\big)$, $L^2$-valued cylindrical Wiener processes $\bar{W}_j$ and $\bar{W}$ with $\bar{W}_j(\bar{\omega})=\bar{W}(\bar{\omega});~j\ge 1$ for  all $\bar{\omega}\in \bar{\Omega}$, independent time-homogeneous Poisson random measures $\bar{N}_j(\bar{\omega})= \bar{N}(\bar{\omega}) );~j\ge 1$ for all $\bar{\omega}\in \bar{\Omega}$, and
$\mathcal{Z}_T$-valued Borel measurable random variables $\bar{u}$ and $\{\bar{u}_j\}_{j=1}^\infty$ such that 
 \begin{align}
 \mathcal{L}(u_{n_j})=\mathcal{L}(\bar{u}_j)~~\text{on}~~\mathcal{Z}_T; \quad \bar{u}_j\goto \bar{u}~~\text{in}~~\mathcal{Z}_T,~~\bar{\mathbb{P}}\text{-a.s.}. \label{con-con-con}
 \end{align}
Moreover, the tuple $\big(\bar{\Omega},
\bar{ \mathcal{F}}, \bar{\mathbb{P}}, \bar{\mathbb{F}}, \bar{W},\bar{ N}, \bar{u}\big)$ is a martingale solution of \eqref{eq:nonlinear-no-control} with initial data $u_0$. Since $\phi\in SC_b(L_w^2)$, by \eqref{con-con-con}, $\bar{\mathbb{P}}$-a.s., $\phi(\bar{u}_j(t)) \goto \phi(\bar{u}(t))$ for any fixed $t$. Moreover, by dominated convergence theorem
\begin{align}
\lim_{j\goto \infty} \bar{\mathbb{E}}\big[ \phi(\bar{u}_j(t)) \big]=  \bar{\mathbb{E}}\big[ \phi(\bar{u}(t)) \big]\,. \label{con-con-con-1}
\end{align}
Since $\mathcal{L}(\bar{u}_j)= \mathcal{L}(u_{n_j})$ on $\mathbb{D}([0,T];L_w^2])$, one has $\mathcal{L}(\bar{u}_j)= \mathcal{L}(u_{n_j})$ on $L_w^2$ and therefore
\begin{align}
 \bar{\mathbb{E}}\big[ \phi(\bar{u}_j(t)) \big]=  \mathbb{E}\big[ \phi(\bar{u}_{n_j}(t)) \big]:= (P_t\phi)(v_{n_j})\,. \label{con-con-con-2}
\end{align}
Let $u$ be a strong solution of \eqref{eq:nonlinear-no-control} with initial data $u_0$. Since $\bar{u}$ is also a martingale solution of \eqref{eq:nonlinear-no-control} with initial data $u_0$, by Lemma \ref{lem:uniqueness-law}, we see that $\mathcal{L}(u)=\mathcal{L}(\bar{u})$ on $\mathcal{Z}_T$, and hence
\begin{align}
(P_t \phi)(v):= \mathbb{E}\big[ \phi(u(t))\big]= \bar{\mathbb{E}}\big[ \phi(\bar{u}(t)) \big]\,. \label{con-con-con-3}
\end{align}
In view of \eqref{con-con-con-1}, \eqref{con-con-con-2} and \eqref{con-con-con-3}, we get
\begin{align*}
\lim_{j\goto \infty} (P_t\phi)(v_{n_j})=\lim_{j\goto \infty} \bar{\mathbb{E}}\big[ \phi(\bar{u}_j(t)) \big]=  \bar{\mathbb{E}}\big[ \phi(\bar{u}(t)) \big]=\mathbb{E}\big[ \phi(u(t))\big]=(P_t \phi)(v)\,.
\end{align*}
Moreover, by uniqueness, the whole sequence converges. This completes the proof of the assertion ${\rm i)}$. 
\vspace{.2cm}

\noindent{Proof of ${\rm ii)}$}. An application of It\^{o}-formula to $u\mapsto \frac{1}{2}\|u\|_{L^2}^2$, integration by parts formula along with Gauss-Green's theorem yields, after taking expectation
\begin{align*}
& \frac{1}{2} \mathbb{E}\big[\|u(t)\|_{L^2}^2\big] - \frac{1}{2}\|u_0\|_{L^2}^2 + \mathbb{E}\Big[ \int_0^t \big\langle {\tt A}(\cdot, u(s), \nabla u(s)), \nabla u(s) \big\rangle\,{\rm d}s\Big] \\
& = \frac{1}{2} \mathbb{E}\Big[ \int_0^t \|{\tt h}(u(s))\|_{\mathcal{L}_2(L^2)}^2\,{\rm d}s\Big] +
\frac{1}{2} \mathbb{E}\Big[ \int_0^t\int_{\tt E} \|\eta(\cdot, u_*(s-);z)\|_{L^2}^2\, m({\rm d}z)\,{\rm d}s\Big] \,.
\end{align*}
In view of the assumption ${\rm ii)}$, \ref{A2}, we have 
\begin{align*}
C_1 \mathbb{E}\Big[\int_0^t \|\nabla u(s)\|_{L^p}^p\,{\rm d}s \Big] \le  \mathbb{E}\Big[ \int_0^t \big\langle {\tt A}(\cdot, u(s), \nabla u(s)), \nabla u(s) \big\rangle\,{\rm d}s\Big] + t \|K\|_{L^1}\,, 
\end{align*}
and therefore
\begin{align*}
& \mathbb{E}\big[\|u(t)\|_{L^2}^2\big] + \mathbb{E}\Big[\int_0^t \big( 2C_1 \|\nabla u(s)\|_{L^p}^p- \|{\tt h}(u(s))\|_{\mathcal{L}_2(L^2)}^2 - \int_{\tt E} \|\eta(\cdot, u_*(s-);z)\|_{L^2}^2\, m({\rm d}z) \big)\,{\rm d}s \Big] \notag \\
& \le  2t \|K\|_{L^1} + \|u_0\|_{L^2}^2 \,.
\end{align*}
Thanks to  the condition \eqref{cond:extra-sigma}, we have,  for any $T>0$
\begin{align*}
\frac{1}{T} \int_0^T \mathbb{E}\big[\|u(t)\|_{L^2}^p\big]\,{\rm d}s \le \frac{1}{\delta} \Big( 2\|K_1\|_{L^1} + C_{hg} + \frac{1}{T}\|u_0\|_{L^2}^2\Big)\,.
\end{align*}
Taking $\nu=\delta_{u_0}$, by Markov inequality, we get for any $T>0$ and $R>0$
\begin{align*}
& \frac{1}{T}\int_0^T (P_{t}^{*} \delta_{u_0}) \big\{ \|x\|_{L^2} >R\big\}\,{\rm d}t = \frac{1}{T}\int_0^T \mathbb{P}\big\{ \|u(t,u_0)\|_{L^2} >R\big\}\,{\rm d}t  \notag \\
& \le \frac{1}{R^p}\, \frac{1}{T}\int_0^T \mathbb{E}\big[ \| u(t)\|_{L^2}^p\big]\,{\rm d}t \le  \frac{1}{ R^p \delta} \Big( 2\|K_1\|_{L^1} +  C_{hg} + \frac{1}{T}\|u_0\|_{L^2}^2\Big)\,.
\end{align*}
The assertion ${\rm ii)}$ follows from the above inequality once we choose $T_0 \ge 0$ and $R$ sufficiently large which depends on $\epsilon, \|u_0\|_{L^2}, \|K_1\|_{L^1}, C_{hg}$ and $\delta$.
This completes the proof of Theorem \ref{thm:existence-invariant-measure}.
\begin{rem}
The result of Theorem \ref{thm:existence-invariant-measure} holds true for SPDE driven by additive noise as  in  \cite{Vallet2021}. Indeed,  in  \cite{Vallet2021}, authors have studied the problem 
 \begin{align}
  {\rm d} u -\mbox{div}_x \big( {\tt A}(x, u, \grad u) + \vec{F}(u)\big)\,{\rm d}t &= {\tt \Phi}\,dW(t), \quad 
  u(0,\cdot)=u_0(\cdot)\,, \label{eq:additive-noise}
 \end{align}
where $W$ is a cylindrical Wiener process in $L^2$ and ${\tt \Phi}$ is a progressively measurable Hilbert-Schmidt operator on $L^2$ with ${\tt \Phi}\in L^2(\Omega; C([0,T];\mathcal{L}_2(L^2)))$, and showed existence  a path-wise unique strong solution $u$ of \eqref{eq:additive-noise} for every $T>0$ such that  $\mathbb{P}$-a.s., $u\in C([0,T;L^2])\cap L^p(0,T;W_0^{1,p})$.
Moreover, Lemmas \ref{lem:uniqueness-law}-\ref{lem:feller-semi-group} and the sequentially weakly Feller semi-group property hold. Therefore, to prove existence of invariant measure, we only need to prove the boundedness in probability. To do so, we apply It\^{o}-formula to the functional $u\goto \|u\|_{L^2}^2$, use ${\rm ii)}$ of \ref{A2}  to have
\begin{align*}
\mathbb{E}\big[\|u(t)\|_{L^2}^2\big] +  2C_1 \mathbb{E}\Big[\int_0^t  \|\nabla u(s)\|_{L^p}^p\,{\rm d}s \Big] \le  2t \|K\|_{L^1} + \|u_0\|_{L^2}^2  + \mathbb{E}\Big[\int_0^t \|\Phi(s)\|_{\mathcal{L}_2(L^2)}^2\,{\rm d}s\Big]\,.
\end{align*}
In view of Poincar{\'e} inequality along with the Sobolev embedding  $W_0^{1,p}\hookrightarrow L^2 $ with $C_P$ and $C_S$ being Poincar{\'e} and Sobolev constants respectively, we have
\begin{align*}
\frac{1}{T}\int_0^T \mathbb{E}\big[ \|u(t)\|_{L^2}^p\big]\,{\rm d}t & \le C_S^p \frac{1}{T}\int_0^T \mathbb{E}\big[ \|u(t)\|_{W_0^{1,p}}^p\big]\,{\rm d}t \le 
C_S^p C_{P}^p 
\frac{1}{T}\int_0^T \mathbb{E}\big[ \|\nabla u(t)\|_{L^p}^p\big]\,{\rm d}t \notag \\
& \le \frac{1}{2C_1 C_S^p C_P^p} \Big\{  \big(2 \|K\|_{L^1} + \mathbb{E}\big[ \|{\tt \Phi}\|_{C([0,T];\mathcal{L}_2(L^2))}^2\big]\big) + \frac{1}{T}\|u_0\|_{L^2}^2 \Big\}\,.
\end{align*}
Hence by Markov inequality, for any $T>0$ and $R>0$
\begin{align}
\frac{1}{T}\int_0^T (P_{t}^{*} \delta_{u_0}) \big\{ \|x\|_{L^2} >R\big\}\,{\rm d}t&= \frac{1}{T}\int_0^T \mathbb{P}\big\{ \|u(t)\|_{L^2} >R\big\}\,{\rm d}t \le \frac{1}{R^p}\, \frac{1}{T}\int_0^T \mathbb{E}\big[ \| u(t)\|_{L^2}^p\big]\,{\rm d}t \notag \\
& \le  \frac{1}{2C_1 R^p C_S^p C_P^p} \Big\{  \big(2 \|K\|_{L^1} + \mathbb{E}\big[ \|{\tt \Phi}\|_{C([0,T];\mathcal{L}_2(L^2))}^2\big]\big) + \frac{1}{T}\|u_0\|_{L^2}^2 \Big\}\,.
\label{inq:final-additive}
\end{align}
Therefore, for fixed $T>0$, r.h.s of \eqref{inq:final-additive} can be made sufficiently small for suitable choice of $R$---which then yields the assertion. 
\end{rem}

\noindent{\bf Acknowledgement:} The first author acknowledges the financial support by CSIR~(09/086(1440)/2020-EMR-I), India and the second author acknowledges the financial support by Department of Science and Technology, Govt. of India-the INSPIRE fellowship~(IFA18-MA119). 

\vspace{.2cm}
\noindent{\bf Data availability:} No data sets were generated during the current study and therefore data sharing is not applicable to this article. 

\vspace{.2cm}
\noindent{\bf Conflict of interest:} The authors have not disclosed any competing interests.

\end{document}